\documentclass[11pt]{amsart}

\usepackage{mathptmx}
\usepackage{helvet}
\usepackage{courier}
\usepackage{graphicx}
\usepackage{multicol}

\usepackage[T1]{fontenc}
\usepackage{euscript}
\usepackage{amsmath}
\usepackage{amsthm}
\usepackage{amssymb}
\usepackage{amscd}
\usepackage{epic}
\usepackage{enumerate}
\usepackage{array}
\usepackage{tabularx}
\usepackage[T1]{fontenc}
\usepackage{calrsfs}
\DeclareMathAlphabet{\pazocal}{OMS}{zplm}{m}{n}

\usepackage{amsfonts,amsthm,amsmath}
\usepackage{amssymb}

\numberwithin{equation}{section}
\numberwithin{equation}{subsection}

\theoremstyle{plain}

\newtheorem{theorem}[equation]{Theorem}
\newtheorem{lemma}[equation]{Lemma}
\newtheorem{proposition}[equation]{Proposition}

\newtheorem{corollary}[equation]{Corollary}

\theoremstyle{definition}

\newtheorem{example}[equation]{Example}
\newtheorem{remark}[equation]{Remark}

\newtheorem{bekezdes}[equation]{}

\newcommand{\bC}{{\mathbb C}}

\newcommand{\bR}{{\mathbb R}}
\newcommand{\bZ}{{\mathbb Z}}

\newcommand{\cV}{{\mathcal V}}
\newcommand{\cA}{{\mathcal A}}
\newcommand{\cW}{{\mathcal W}}
\newcommand{\cE}{{\mathcal E}}
\newcommand{\calS}{{\mathcal S}}

\newcommand{\cL}{{\mathcal L}}

\newcommand{\calL}{{\mathcal L}}

\newcommand{\C}{{\calc}}

\newcommand{\fI}{\mathfrak{I}}

\newcommand{\bt}{{\bf t}}

\newcommand{\calF}{{\mathcal F}}

\newcommand{\Z}{\mathbb{Z}}
\newcommand{\Q}{\mathbb{Q}}
\newcommand{\R}{\mathbb{R}}

\newcommand{\tx}{\tilde{X}}

\def\C{\mathbb C}
\def\Q{\mathbb Q}
\def\R{\mathbb R}

\def\Z{\mathbb Z}

\newcommand{\ra}{\rightarrow}

\newcommand{\tX}{\widetilde{X}}

\newcommand{\tC}{\tilde{C}}

\newcommand{\fh}{\mathfrak{h}}
\newcommand{\D}{\Delta}

\newcommand{\ev}{\varepsilon}
\newcommand{\cO}{{\mathcal O}}

\usepackage{tikz-cd}

\begin{document}

\author{L\'aszl\'o Koltai}
\address{HUN-REN Alfr\'ed R\'enyi Institute of Mathematics,
Re\'altanoda utca 13-15, H-1053, Budapest, Hungary
\newline
 \hspace*{4mm} ELTE - Faculty of Science, Dept. of Geometry, P\'azm\'any P\'eter s\'et\'any 1/A, 1117 Budapest, Hungary}
\email{koltai.laszlo@renyi.hu }

\author{Tam\'as L\'aszl\'o}
\address{Babe\c{s}-Bolyai University, Str. Mihail Kog\u{a}lniceanu nr. 1, 400084 Cluj-Napoca, Romania \newline
\hspace*{4mm} Simion Stoilow Institute of Mathematics of the Romanian Academy, Bucharest, Romania}
\email{tamas.laszlo@ubbcluj.ro}
\author{Andr\'as N\'emethi}
\address{HUN-REN Alfr\'ed R\'enyi Institute of Mathematics,
Re\'altanoda utca 13-15, H-1053, Budapest, Hungary \newline
 \hspace*{4mm} ELTE - Faculty of Science, Dept. of Geometry,
 P\'azm\'any P\'eter s\'et\'any 1/A, 1117 Budapest, Hungary \newline \hspace*{4mm}
  BBU - Babe\c{s}-Bolyai Univ., Str, M. Kog\u{a}lniceanu 1, 400084 Cluj-Napoca, Romania
   \newline \hspace*{4mm}
BCAM - Basque Center for Applied Math.,
Mazarredo, 14 E48009 Bilbao, Basque Country – Spain}
\email{nemethi.andras@renyi.hu }

\thanks{The authors are partially supported by NKFIH Grant ``\'Elvonal (Frontier)'' KKP 144148. The second author acknowledges the support of the 'J\'anos Bolyai Research Scholarship' of the Hungarian Academy of Sciences, and the support of the project “Singularities and Applications” - CF 132/31.07.2023 funded by the European Union - NextGenerationEU - through Romania’s National Recovery and Resilience Plan.}

\title{Multiplier ideals of normal surface singularities}




\begin{abstract}
We study the multiplier ideals and the corresponding jumping numbers and multiplicities $\{m(c)\}_{c\in \R}$ in the following context: $(X,o)$ is a complex analytic normal surface singularity,
${\mathfrak a}\subset \cO_{X,o}$ is an ${\mathfrak m}_{X,o}$--primary ideal,
$\phi:\tX\to X$ is a log resolution of $\mathfrak{a}$ such that $\mathfrak{a}\cO_{\tX}=\cO_{\tX}(-F)$,
for some nonzero effective divisor $F$ supported on $\phi^{-1}(0)$.
A priori,  $\{m(c)\}_{c\in \R}$ depends on the Hilbert function associated with the resolution and $F$.

We prove  that $\{m(c)\}_{c>0}$ is combinatorially computable from $F$ and the resolution graph $\Gamma$
of $\phi$, and we provide several formulae. We also extend Budur's result (valid for
$(X,o)=(\C^2,0)$), which makes an identification of $\sum_{c\in[0,1]}m(c)t^c$ with a certain Hodge spectrum. In our general case  we  use Hodge spectrum with coefficients in a mixed Hodge module.

We show that $\{m(c)\}_{c\leq 0}$ usually depends on the analytic type of $(X,o)$. However, for some distinguished analytic types we determine it concretely. E.g., when $(X,o)$ is
weighted homogeneous (and $F$ is associated with the central vertex), we recover $\sum_cm(c)t^c$ from the Poincar\'e series of $(X,o)$.  Also, when $(X,o)$ is a splice quotient then we recover  $\sum_cm(c)t^c$
from the multivariable topological Poincar\'e (zeta) function of $\Gamma$.

\end{abstract}

\maketitle

\reversemarginpar

\section{Introduction}

Assume that $(X,o)$ is a complex analytic normal surface singularity, let $\cO_{X,o}$ be its
local algebra with maximal ideal ${\mathfrak m}_{X,o}$. Let us also fix an
${\mathfrak m}_{X,o}$--primary ideal ${\mathfrak a}\subset \cO_{X,o}$. Then there exists a log resolution
$\phi:\tX\to X$  and a nonzero
integral cycle $F$ supported on the exceptional curve $\phi^{-1}(0)=E$ such that ${\mathfrak a}\cO_{\tX}=\cO_{\tX}(-F)$. In fact, in our study we will start  with some log resolution 
$\phi$ and  such a nonzero antinef cycle $F$.

We denote by $Z_K$ the anticanonical cycle associated with $\phi$,  namely $-c_1(\Omega^2_{\tX})=-c_1(\cO_{\tX}(K_{\tX}))\in H^2(\tX,\Z)$,
identified with a rational cycle (via the adjunction relations).
We consider for any $c\in\R$ the cycle
$L(c):=\lfloor Z_K+cF\rfloor$. Then the tower of {\it multiplier ideals} is defined as
$\fI(c,F):=\phi_*\cO_{\tX}(-L(c))\subset \cO_{X,o}$. They measure the complexity of the singularity $(X,o)$ compatibly with $F$ (or with $\mathfrak{a}$). They were introduced (at least implicitly)
in \cite{[21],[25],[33],[34]} and in the present form in
\cite{Ein,Laz2}.
Since then they were intensively studied, see \cite{budur,Carrami,Carrami2,Galindo,Galindo2,Howald,Hyry,Jarvilehto,Montaner,Naie,Tucker} with all their connections with different aspects of singularities and singular map-germs. In fact, the definition and study
extends to any dimension, see e.g. \cite{Ein,Laz2,Pande} and the references therein.

It turns out that  for any $c\in \R$, the ideal $\fI(c,F)\subset \cO_{X,o}$
has finite codimension, $\fI(c,F)\subset \fI(c',F)$ for $c'< c$, and $\fI(c,F)=\cO_{X,o}$ for
$c\ll 0$.  The number $c$ is called {\it jumping number} with {\it jumping multiplicity} $m(c)$
 if $m(c):= \dim\,
\fI(c-\epsilon ,F)/\fI(c, F)\not=0$ for $0<\epsilon\ll 1$.
One verifies that in any compact interval there are only finitely many jumping numbers and they are all rational.

Compared with the earlier literature in the subject we make two major extensions:

\vspace{2mm}

\noindent {\bf (A)} \ We do not assume any restriction regarding the surface singularity $(X,o)$, e.g. we do not assume that the link is a rational homology sphere (i.e. the dual graph might have loops and the irreducible exceptional divisors might have higher genera).
Also, we do not restrict the analytic type of $(X,o)$: on the  supporting topological type the analytic type might vary (which in the non-rational or non-minimally elliptic cases
might have considerable effects). We emphasize that in general the classification
of such analytic types (or the description of the corresponding  moduli spaces)
is still an open problem. Under this generality the general analytic tools
(e.g. vanishing theorems) are more complex (and the zones what they do not cover
are wider).

\vspace{2mm}

\noindent {\bf (B)} \ In the earlier studies only the cases $c>0$ were treated.
Our study of the cases $c\leq 0$ opens  new  perspectives. If $(X,o)$ is smooth or log terminal then $m(c)=0$ for $c\leq 0$, however, in the general case (even for rational $(X,o)$) the negative zone appears as a  very challenging part.

In fact  (and this is one of the main messages  of the note),
we show that the study of  $\{m(c)\}_{c\in \R}$ (equivalently, 
the Poincar\'e series of multiplier ideals $\sum_{c\in\R}m(c)t^c$ introduced first in \cite{Galindo}) breaks into two major parts with rather different aspects.

\vspace{2mm}

{\bf (i)} \ {\bf (the $c>0$ zone)} \ We also call this the {\it Riemann--Roch zone},
since in this case  we determine $\sum_{c>0}m(c)t^c$ completely topologically  from the resolution graph $\Gamma$ of $\phi$ and the cycle $F$ (without any restriction regarding the link or the analytic type of $(X,o)$).

We provide several combinatorial formulae. The first one,  Proposition \ref{lem:chi},
is a generalization of
expressions obtained in \cite{Carrami2,Carrami,Tucker} proved in those articles for rational singularities.

We also compute $m(c+1)-m(c)$ (see Corollary \ref{cor:m0}), thus  we recover $\sum_{c>0}m(c)t^c$ from
$\sum_{c\in (0,1]}m(c)t^c$ and $(D_c,F)$, where $D_c=L(c)-L(c-\epsilon)$
($0<\epsilon\ll 1$) is the {\it jumping
divisor} (introduced in \cite{Carrami2}), computable combinatorially.
(This fact was also shown  for rational $(X,o)$ in  \cite{Carrami}.)

Hence, the first interval and
$\sum_{c\in (0,1]}m(c)t^c$ has a distinguished role.

Regarding this we stress out two facts. Firstly, Budur in \cite{budur}
identified $\sum_{c\in (0,1]}m(c)t^c$ with the Hodge spectrum ${\rm Sp}_{(0,1]}(f,t)$
of a function germ $f$, where in his case,  $(X,o)=(\C^2,0)$ is smooth, $f:(\C^2,0)\to
(\C,0)$ is a function germ with an isolated singularity, $\phi$ is a log embedded resolution and $F$ is the divisor ${\rm div}_{\phi,E}(f)$ of $\phi^*(f)$ along $\phi^{-1}(0)=E$. This was reproved in \cite{Carrami} as well.
(For the theory of mixed Hodge structures and of the spectrum associated with isolated hypersurface singularities see the work of
Steenbrink, e.g. \cite{steenbrink,steenbrink2} or \cite{SSS}.)

In this note we succeed to generalize this statement. The first step is that we rewrite
$J_{[0,1]}(t)=\sum_{c\in[0,1]}m(c)t^c$ in a different way, in terms of $\Gamma$ and the multiplicities of $F$, for any $(X,o)$ and any $F$ (cf.  Theorem \ref{th:main_mult}).
This is the second combinatorial formula of the jumping multiplicities.

This formula will resonate perfectly with the Hodge spectrum computations of a generic element $f\in\mathfrak{a}$. Indeed, if $(X,o)$ is numerically Gorenstein
(that is, $Z_K$ is an integral cycle),  and
$f$ is an analytic function gems $(X,o)\to (\C,0)$ with isolated singularities  and  ${\rm div}_{\phi,E}(f)=F$,
we prove the generalization of the statement of Budur:
$J_{[0,1]}(t)={\rm Sp}_{[0,1]}(f,t)$, cf. Corollary \ref{cor:NS}. However, when $\Gamma $ is not numerically Gorenstein,  $\{Z_K\}\not=0$, then we need to use a more  general spectrum formula
established in  \cite{NS}. In \cite{NS}  the `usual' spectrum is replaced by the spectrum of $f$ with coefficients in a mixed Hodge module ${\mathcal M}$
of $(X,o)$,
such that ${\mathcal M}$ restricted to $X\setminus \{o\}$ is a variation of Hodge structure whose underlying representation is abelian.
Indeed, under the assumption that the link is a rational homology sphere,
from $Z_K$ we construct a character of the first homology of the link, hence a
rank one variation of Hodge structure ${\mathbb V}$ on $X\setminus \{o\}$, and we prove the identity
$J_{[0,1]}(t)={\rm Sp}_{[0,1]}(f, j_* {\mathbb V},t)$ (where
$j:X\setminus \{o\}\subset X$ is the inclusion), see Theorem \ref{th:NS2}.
The construction of ${\mathbb V}$
can be done using the canonical cyclic covering of $(X,o)$ (associated with $Z_K$,
or in $\Q$--Gorenstein case associated with the canonical divisor  $K_{\tX}$, as well). We emphasize again, though a priori both
objects $J_{[0,1]}(t)$ and ${\rm Sp}_{[0,1]}(f, j_* {\mathbb V},t)$ are analytical, we provide combinatorial formulae for them. It would be interesting to compare these results with the general (Hodge) theoretical construction
of \cite{schnell} valid for complex manifolds, and with the Bernstein-Sato polynomial connection from \cite{ELSV}.

The second new aspect regarding  $\sum_{c\in (0,1]}m(c)t^c$ is the following connection.
Consider  again a function $f:(X,o)\to (\C,0)$ with an isolated singularity such that $F={\rm div}_{\phi,E}(f)$. Set $(C,o):=(f^{-1}(0),o)$ and let
 $\delta(C,o)$ be its  delta invariant.
Then  $\sum_{c\in (0,1]}m(c)=
\delta(C,o)+(\{Z_K\},F)$.

\vspace{2mm}

{\bf (ii)} \ {\bf (the $c\leq 0$ zone)} \
We show by examples that the multiplicity  $m(0)$ cannot be determined
topologically. However, we characterize  it  completely in terms of
the analytic type of $(X,o)$.  The analytic contribution is a term $\mathfrak{h}^0\in\{0,1\}$
characterized by an $h^0$--sheaf cohomology (see \ref{s:cpoz}).

The $c\leq 0$ zone can be called the
`$h^0$--zone', where the $h^1$--vanishing theorems (hence the Riemann--Roch expressions) 
 cannot be applied.
In this zone $m(c)$ usually cannot be computed form $\Gamma$ and $F$. We exemplify this fact
by concrete examples when we fix $\Gamma$ and $F$ and we consider several different analytic realizations, and we compute the $m(c)$'s. Interestingly enough, already the numerical invariant $\sum _{c\leq 0}m(c)$
 measures the analytic type and it is connected to the most important analytic invariant of $(X,o)$, the geometric genus $p_g(X,o)$. More precisely, we prove that
 $\sum _{c\leq 0}m(c)=p_g(X,o)+\chi(\lfloor Z_K\rfloor)$.
 Already when $(X,o)$ is rational (i.e. $p_g(X,o)=0$),  $\sum _{c\leq 0}m(c)$ can be arbitrarily large
 (see e.g. Example \ref{ex:sum}).

 On the other hand, even for  this zone, we provide concrete formulae valid for certain
 {\it distinguished analytic types}.

 \vspace{2mm}

 $\bullet$ \ In the case of {\it a weighted homogeneous germs}
 $(X,o)$  with local graded algebra
 $\oplus_{\ell\geq 0}R_\ell$
 (and $F$ is associated with the central vertex),
 we show the correspondence $\dim(R_\ell)=m(c) $ with $
 c=(\ell-\mathfrak{r})/\alpha$ (for notations and precise statement see subsection \ref{ss:we}). For $\ell>{\mathfrak r}$
 (corresponding to $c>0$) $\dim(R_\ell)$ is topological computable from the Seifert invariants of the link, however in general it is not.
 But, if the link is a rational homology sphere,  then by Dolgachev-Pinkham-Demazure formulae
 all the terms for any $\ell$ are topological. (For the theory of weighted homogeneous germs see e.g. \cite{pinkham}.)

 $\bullet$ \ In the next case we assume that
 the link is a rational homology sphere, and we fix a resolution graph $\Gamma$, and we assume that the analytic type of $(X,o)$ is {\it splice quotient} with respect to
 $\Gamma$ (a family introduced by Neumann and Wahl, cf. \cite{NW}).
 Then in Proposition \ref{PZ}  we provide a combinatorial formula
 for  $\sum_{c\in \R}m(c)t^c$ in terms of $\{L(c)\}_c$ and
 the multivariable topological Poincar\'e (zeta)
 function $Z(\bt)$  of $\Gamma$. This includes e.g. all the cases when $(X,o)$
 is rational or minimally elliptic (with $\Gamma$ minimal). (For more on $Z(\bt)$ see \cite{CDG,CDGEq,Nkonyv}.)

 But we emphasize again that in general, the computation of $\{m(c)\}_{c\leq 0}$
 is of analytic nature, hence it runs in parallel with the classification of the analytic types
 of $(X,o)$ supported of a fixed topological type identified by $\Gamma$. We will return back to the study
  of these analytic aspects in the forthcoming manuscript.

\section{Preliminaries}

\subsection{Notations regarding resolutions and resolution graphs}

\bekezdes
Let us fix a normal surface singularity $(X,o)$ and one of its log (good) resolutions  $\phi:\widetilde{X}\to X$.
Let  $\Gamma$ be the dual resolution graph of $\phi$, we denote its vertices by  $\mathcal{V}$ and its edges by $\mathcal{E}$.
  We write $E$ for the irreducible exceptional curve $\phi^{-1}(o)$, and let $\{E_i\}_{i\in\mathcal{V}}$ be its irreducible components.

The lattice $L:=H_2(\widetilde{X},\mathbb{Z})$ is a
free $\Z$--module generated by the classes of  $\{E_i\}_{i\in\mathcal{V}}$
and it is  endowed
with the natural  negative definite intersection form  $(\,,\,)$.
 The dual lattice is $L'={\rm Hom}_\Z(L,\Z) \simeq\{
l'\in L\otimes \Q\,:\, (l',L)\in\Z\}\subset L\otimes \Q$.
We use the same symbol $(\,,\,)$ for the extended intersection
form to $L\otimes \Q$. The overlattice
$L'$   is generated over $\Z$
by the (anti)dual classes $\{E^*_i\}_{i\in\mathcal{V}}$ defined
by $(E^{*}_{i},E_j)=-\delta_{ij}$ (where $\delta_{ij}$ stays for the  Kronecker symbol).
 $L'$ can also  identified with $H^2(\tX,\Z)$ (the target of
 the first Chern class $c_1:{\rm Pic}(\tX)\to H^2(\tX,\Z)$).

We denote by $M$ the link of $(X,o)$. It is known that $M$ is a rational homology sphere if and only if each $E_i$ is rational and $\Gamma$ is a tree.
It is an integral homology sphere if additionally $\det\,(\,,\,)=\pm 1$.
More generally, if $g_i$ denotes the genus of $E_i$, $g:= \sum_i g_i$,
 and $h$ is the first Betti number of the topological realization of $\Gamma$ (the number of independent
cycle in the graph), then the first Betti number of $M$ is $2g+h$, and
 $ {\rm Tors}\, H_1(M,\mathbb{Z})\simeq L'/L$.

There is a natural partial ordering of $L\otimes \mathbb{Q}$: we write $l_1'\geq l_2'$ if
$l_1'-l_2'=\sum _i a_iE_i$ with every  $a_i\geq 0$.
We set $L_{\geq 0}=\{l\in L\,:\, l\geq 0\}$ and
$L_{>0}=L_{\geq 0}\setminus \{0\}$.

We define the rational {\it Lipman cone} as $\calS_\Q:=\{l'\in L\otimes \Q\,:\, (l', E_i)\leq 0 \ \mbox{for all $i$}\}$, and we  also
set $\calS':=\calS_\Q\cap L'$ and  $\calS:=\calS_\Q\cap L$.
As a monoid $\calS'$  is generated over $\bZ_{\geq 0}$ by $\{E^*_i\}_i$.
 If $s'\in\calS_\Q\setminus \{0\}$, then
all its $E_i$--coordinates  are strictly positive (see \cite[Corollary 2.1.19]{Nkonyv}).

Let $K_{\tX}$ be the {\it canonical divisor } of $\tX$ (well--defined up to a linear equivalence), that is,  $\cO_{\tX}(K_{\tX})\simeq\Omega^2_{\tX}$, the sheaf of holomorphic 2--forms. We denote
its  first Chern class
$c_1(\cO_{\tX}(K_{\tX}))\in  H^2(\tX,\Z)\simeq L'$ by $-Z_K$; $Z_K$ is called
the  {\it (anti)canonical cycle} of $\phi$.
The (rational) cycle  $Z_K\in L'$ can also  be determined combinatorially  from
$(\,,\,)$ by the
{\it adjunction formulae}
$(Z_K, E_i)=(E_i,E_i)+2-2g_i$ for all $i\in\mathcal{V}$.

We call $(X,o)$ {\it numerically Gorenstein } if $Z_K\in L$, i.e. if all its $E_i$--coefficients are integral (this condition is independent of the choice of $\phi$).
We call $(X,o)$ {\it Gorenstein} if $Z_K\in L$ and one can choose $K_{\tX}$ as
$-Z_K$, i.e. $\cO_{\tX}(K_{\tX}+Z_K)\simeq \cO_{\tX}$. In a different language,
$(X,o)$ is numerically Gorenstein (resp. Gorenstein)
exactly when  the line bundle $\Omega^2_{X\setminus \{o\}}$ is topologically (resp. analytically) trivial.
(For a different characterization see also \ref{bek:qg}.)

Finally,  we set for any $l'\in L'$ the Riemann-Roch expression $\chi(l')=(Z_K-l',l')/2$.  If $l\in L_{> 0}$, then $\chi(l)=h^0(\cO_l)-h^1(\cO_l)$. E.g., $\chi(E)=1-g-h$.
As usual,  we also write $\chi(\cO_l(D))$ for $h^0(\cO_l(D))-h^1(\cO_l(D))$,
which,  by  Riemann-Roch, equals $\chi(l)+(D,l)$ ($D\in L$).

For any $\bR$-divisor $D=\sum_i d_i E_i$ of $\widetilde{X}$ ($d_i\in \bR$) we define its integral part
 $\lfloor D \rfloor$ and fractional part $\{ D\}$ by applying the corresponding operation to each  coefficient $d_i$.

For more details see e.g.  \cite{NOSz,Nkonyv}.
\bekezdes \label{bek:1.1.2}
Assume that  $(C,o)$  is a  reduced Weil divisor in $(X,o)$
(i.e. it is  an isolated  curve singularity in $(X,o)$).
Then we can consider a log embedded resolutions of the pair $(C,o)\subset (X,o)$, that is, we require that in the log resolution
the union of  $E$ with the strict transform $\tilde{C}$ of $(C,o)$ form a simple normal crossing divisor.
Moreover, we can consider the corresponding embedded resolution graph too. Usually, the strict  transform will be denoted by arrows on the graph.
Their index set will be denoted by $\cA$ corresponding to the
irreducible decomposition $\cup_{a\in\cA}C_a$ of $(C,o)$.

Once an embedded resolution and  the strict transform $\tC \subset \tX$ is fixed,  the embedded topological type of the pair $(C,o)\subset (X,o)$
is basically coded in the information that how many components of $\tC$ intersect each $E_i$, that is, in the intersection numbers $(\tC, E_i )_{\tX}$ . Then we define the rational cycle $l'_C\in \calS'\setminus \{0\} \subset  L'_{>0}$,  associated with $C$ and $\phi$,
via the linear system of equations
 $(l'_C,E_i) + (\tC,E_i)_{\tX} = 0$ for every
 $i\in\mathcal{V}$. (Hence, $l'_C +\tC$, as a rational divisor of $\tX$,  is numerically trivial.
 It is the total transform of $C$,  denoted by $\phi^*C$.)
If the arrow $a$  is supported by the vertex  $i(a)$, then  $l'_C=\sum_{a\in\cA} E_{i(a)}^*$.

If  $(C,o)$ is cut out by a germ $f:(X,o)\to (\mathbb{C},0)$ such that $(C,o)=(f^{-1}(0),o)$ is reduced (that is, $f$ defines an  isolated singularity), then we say that $(C,o)$  is
 a Cartier divisor of $(X,o)$.  In such a case $\phi^*C=
{\rm div}_{\tX}(\phi^*f)$.  Hence, in this case,   $l'_C$ is integral.
It will also be denoted by  ${\rm div}_{\phi,E}(f)$.

\bekezdes\label{bek:1}
Next, we fix a cycle $F\in \calS\setminus \{0\}$. Some  geometrical situations 
(motivated by the present note) how such a cycle  might  appear is listed below.

\vspace{1mm}

\noindent {\bf (A)} Let $\mathfrak{m}_{X,o}$ denote the maximal ideal of the local ring $\cO_{X,o}$ of $(X,o)$.
Let $\mathfrak{a}$ be an $ \mathfrak{m}_{X,o}$--primary   ideal of $\cO_{X,o}$. Then there exists
a log resolution $\phi$ such that $\mathfrak{a}\cdot \cO_{\widetilde{X}}=\cO_{\widetilde{X}}(-F)$ for some $F\in\calS\setminus \{0\}$.

\noindent {\bf (B)} Let $(C,o)\subset (X,o)$ be  a reduced  embedded curve germ, and $\phi$ an embedded good resolution. Then $l'_C\in\calS'\setminus \{0\}$, hence a convenient multiple of it  is in $\calS\setminus \{0\}$. Cf. also with Remark \ref{rem:multiple}.

\noindent {\bf (C)} Let $f:(X,o)\to (\mathbb{C},0)$ be a germ, then in any good embedded resolution $F:={\rm div}_{\phi,E}(f) \in \calS\setminus \{0\}$.

If $\mathfrak{a}$  is as in {\bf (A)}, and we fix a resolution $\phi$  with $\mathfrak{a}\cdot \cO_{\widetilde{X}}=\cO_{\widetilde{X}}(-F)$, then for any generic element
$f$ of  $\mathfrak{a}$ one has ${\rm div}_{\phi,E}(f)=F$.

\noindent {\bf (D)} Conversely, let us fix a semigroup element $F\in \calS\setminus \{0\}$, a divisor of a resolution of some analytic germ $(X,o)$.
If $(X,o)$ is rational then there exists a germ $f$ such that
${\rm div}_{\phi,E}(f)=F$, but for an arbitrary singularity $(X,o)$ this is not the case. However, if we fix the topological type of $(X,o)$, identified by $M$ or by some resolution graph
$\Gamma$, then there exists an analytic type, say $(X_F,o)$ supported on $M$ and a germ $f_F:(X_F,0)\to (\mathbb{C},0)$ such that ${\rm div}_{\phi_F,E_F}(f_F)=F$.

\noindent {\bf (E)} The correspondence $\mathfrak{a}\mapsto F$ given by
 $\mathfrak{a}\cdot \cO_{\widetilde{X}}= \cO_{\widetilde{X}}(-F)$ (via some good resolution)
is not injective. E.g., if we start with such an $\mathfrak{a}$ and $F$, then  $\phi_*\cO_{\widetilde{X}}(-F)$ is the integral closure $\bar{\mathfrak{a}}$ of
 $\mathfrak{a}$ in $\cO_{X,o}$, and  $\bar{\mathfrak{a}}\cdot \cO_{\widetilde{X}}=
 \mathfrak{a}\cdot \cO_{\widetilde{X}}=
 \cO_{\widetilde{X}}(-F)$
 too. Moreover, for any other ideal $\mathfrak{b}$ the identity
 $\bar{\mathfrak{a}}=\bar{\mathfrak{b}}$ is equivalent with
 $\mathfrak{b}\cdot \cO_{\widetilde{X}}=\mathfrak{a}\cdot \cO_{\widetilde{X}}$.

\subsection{The multiplier ideals}

\bekezdes {\bf The definition of multiplier ideals,
jumping numbers and jumping multiplicities.}
We fix a normal surface singularity $(X,o)$, a good resolution $\phi$, and a cycle $F\in \calS\setminus \{0\}$.

Following \cite{Ein,Laz2}
we define the {\it multiplier ideal}  associated with $F$ and a  real number
$c\in \bR$ as
$$\fI(c,F)= \phi_*\cO_{\widetilde{X}}(-\lfloor Z_K+cF \rfloor)\subset \cO_{X,o}.$$
For simplicity, we denote  the divisor $\lfloor Z_K + c F\rfloor$ by $L(c)$.
Though $Z_K$ and $F$ are combinatorial cycles  associated with the graph $\Gamma$, the ideals $ \fI(c,F)$ might depend essentially on the analytic structure of the germ $(X,o)$.
 Since $\phi:\tX\setminus E\ra X\setminus o$ is an isomorphism and $L(c)$ is supported on $E$,
we obtain  that $\fI(c,  F)= \cO_{X,o}$ whenever $L(c)\leq 0$,  and $\fI(c,F)\subset \cO_{X,o}$ is an $\mathfrak{m} _{X,o}$--primary ideal otherwise.
Also, $ \phi_*\cO_{\widetilde{X}}(-L(c) )= \phi_*\cO_{\widetilde{X}}(-\max\{0,L(c)\} )$ (see e.g. Lemma \ref{lem:PN}).
 As a corollary we get that
$$ \dim(\cO_{X,o}/\fI(c,F))< \infty $$
If $c\leq c'$, then
 $\fI(c ,F)  \supseteq \fI(c' ,F)$.
Whenever we have strict inclusion $$ \fI(c-\ev,F) \supsetneq \fI(c,F)$$
for arbitrarily small $\ev >0$, we say that $c$ is a {\it jumping number}, and its
{\it jumping multiplicity} is:
$$ m(c):=\dim(\fI(c-\ev,F)/\fI(c,F)).$$
Define  the jumping divisor corresponding to $c\in \bR$ (cf. \cite{Carrami2,Carrami,Tucker}) as
$$ D_c:=L(c)-L(c-\ev) \ \ \ (1\gg\ev>0). $$
Obviously, $D_c$ is a reduced divisor and it is `periodical': $D_{c+1}=D_c$.
In fact,  if $Z_K=\sum k_i E_i$ and $F=\sum m_i E_i$, then $E_i$ is in the support of  $D_c $ precisely when $k_i+cm_i\in \bZ$.
This shows that in any bounded  interval there are only finitely many values $c$ with $D_c\not=0$, hence with $m(c)>0$. Note that $D_c\not =0$ usually does not imply  $m(c)>0$ (see e.g. Example \ref{ex:rat}).
We also define
$$\D_c:= \{Z_K+cF\}.$$
Then
   any $E_i$  is either in the support of $D_c$ or of $\D_c$, but not in both.

Since $F\in \calS\setminus \{0\}$, hence all the coefficient of $F$ are strict positive,
if $c\ll0$ then $Z_K+cF\leq 0$,  and
$ \fI(c ,F)=\cO_{X,o}$. In particular, $m(c)=0$ for $c\ll 0$. Thus, in any  interval
of type $(-\infty, c_0]$ there are only finitely many jumping numbers $c$.

\begin{remark}\label{rem:multiple}
The assumption that $F$ is integral is just a convenient choice.
In all our arguments regarding the cycle $F$,
we can equally  start with any $F\in\calS_\Q\setminus\{0\} $, or,  we can replace the integral  $F$ by any $rF$, where $r\in\Q_{>0}$. Indeed,
the effect of the replacement  $F$ by $rF$  is
$m_{rF}(c)=m_{F}(rc)$.

\end{remark}

\section{Jumping multiplicities and the Hilbert function}

\subsection{Multiplier ideals and the Hilbert function}
Let us start with the following (known) lemma.
\begin{lemma}\label{lem:PN} Fix some $l\in L$ and write it as $l= P-N$, with
$P,N\in L_{\geq 0}$ in such a way  that in their support they have no common $E_i$--component. Then

(a) $H^0(\tX, \cO_{\tX}(-P))\to H^0(\tX, \cO_{\tX}(-l))$ is an isomorphism,

(b)  $h^1(\cO_{\tX}(-P))-h^1(\cO_{\tX}(-l))=\chi(N)-(N,l)$.
\end{lemma}
\begin{proof} Consider the short exact sequence of sheaves
$0\to \cO_{\tX}(-P)\to \cO_{\tX}(-l)\to \cO_N(-l)\to0$.

Then
 $(c_1(\cO_{\tX}(-P)),E_i)=(-P,E_i)\leq 0$ for all $E_i$ in the support of $N$.
Therefore  by the  $h^0$-vanishing theorem (see \cite[Theorem 6.4.2.]{Nkonyv}),
the dual of the Grauert-Riemenschneider vanishing,
 we obtain  $H^0(N,\cO_N(-P+N))=0$. Then use the
 long cohomological exact sequence.
\end{proof}
This shows that for any $l\in L$ we have
$H^0(\tX,\cO_{\tX}(-l))=H^0(\tX,\cO_{\tX}(-\max\{0,l\}))$, which  is canonically a subspace of
$H^0(\tX,\cO_{\tX})$.
The Hilbert function $\mathfrak{h}(l) $  (associated with the divisorial filtration of the resolution) is defined as
$$L\ni l\mapsto \mathfrak{h}(l):=\dim\, H^0(\tX,\cO_{\tX})/H^0(\tX,\cO_{\tX}(-\max\{0,l\}))=
\dim \, \cO_{X,o}/ \phi_* \cO_{\tX}(-l).$$

Note that
$\mathfrak{h}(L(c))=
\dim \cO_{X,o}/\fI(c,F)$, hence
\begin{equation}\label{eq:m}
m(c)=\mathfrak{h}(L(c))-\mathfrak{h}(L(c-\ev)).\end{equation}

\bekezdes {\bf Some general comments regarding the formula (\ref{eq:m}).}
Usually, under the assumption that we fix the  combinatorial data $\Gamma$,
the Hilbert function $\mathfrak{h}$ might depend
essentially on the choice of the analytic type of $(X,o)$ .
However, for any rational singularity, for any fixed resolution with graph $\Gamma$, $\mathfrak{h}(l) $ can be determined from $\Gamma$ for any $l\in L$. In fact, this is true for any
singularity  whose link is a rational homology sphere and its analytic type is {\it splice quotient}.
This family of normal surface singularities was introduced by Neumann and Wahl \cite{NW}
and it includes (among others) all rational,  minimally elliptic and weighted homogeneous singularities
(with $\Q HS^3$ link).
For such germs, $\mathfrak{h}(l)$ follows combinatorially  via the following two steps:
by \cite[(8.2.10)]{Nkonyv} $\mathfrak{h}(l)$ can be computed from the multivariable (analytic)
Poincar\'e series $P_0({\mathbf t})=\sum _l \mathfrak{p}(l)\bt^l$ of $(X,o)$ as
$\mathfrak{h}(l) =\sum_{\tilde{l}\not\geq l} \mathfrak{p}(\tilde{l})$
(this is true for any normal surface singularities), while
 by  \cite[Corollary 8.5.35]{Nkonyv} $P_0({\mathbf t})=Z_0({\mathbf t})$, where $Z_0({\mathbf t})$
is the topological multivariable Poincar\'e series (or zeta   function) associated with $\Gamma$, for its expression see Proposition below, (at this step the splice quotient assumption is crucial).
(For the definition and certain properties of analytic and topological multivariable
Poincar\'e series see also
\cite{CDG,CDGEq,CDGc}.)

In particular, in principle, for all splice quotient  singularities (and for any
$F\in \mathcal{S}\setminus \{0\}$) all the jumping multiplicities  $\{m(c)\}_{c\in \R}$ are  computable from $\Gamma$ and $F$.
Here is the precise  statement.

\begin{proposition}\label{PZ}
Assume that the link of $(X,o)$ is a rational homology sphere and $(X,o)$ is splice quotient
(associated with its resolution graph $\Gamma$). Consider free variables $\bt=\{t_i\}_{i\in\cV}$
indexed by the set of vertices, and for any $l'=\sum l'_iE_i\in L'$ write $\bt^{l'}=\prod_i t_i^{l'_i}$.
Let $Z(\bt)=\sum_{l'\in \calS'}\mathfrak{p}(l')\bt^{l'} $
be the  Taylor expansion at the origin of the rational function
    $$ \prod_{i\in\cV}(1-\bt^{E^*_i})^{\kappa_i-2},$$
    where $\kappa_i$ is the valency of the vertex $i$. Let $Z_0(\bt)$ be its integral part, namely
    $Z_0(\bt)=\sum _{l\in\calS} \mathfrak{p}(l)\bt^l\in \Z[\calS]$.
Then for any $c\in \R$ one has
$$m(c)=\sum_{l\geq {L(c)-D_c}, \ l\not\geq L(c)}\ \  \mathfrak{p}(l).$$

\end{proposition}
However, we wish to emphasise, that in the non-rational splice quotient cases, if we keep the topological type (say $\Gamma$) and we modify the analytic structure of $(X,o)$ (into a non-splice quotient one) then
the Hilbert function and the jumping multiplicities might  change.

But, in the rational case (since any analytic type is splice quotient)
 independently of the analytic realizations of the topological type, the Hilbert function and $m(c)$ are  combinatorial expressible from $\Gamma$ and $F$ (via the above recipe).

We also observe  that in the general  splice quotient case the explicit computation via the above recipe can be rather technical and computational.

In this note, for $c\geq 0$ and for arbitrary $(X,o)$
we run  a completely different approach (inspired from mixed Hodge theory). In fact, for $c>0$ we will be able to determine
$m(c)$ {\it completely combinatorially} for any $(X,o)$ and $F$, independently of the analytic structure of $(X,o)$,
and we provide explicit formulae. The number $m(0)$ might depend on  the analytic structure, but   we determine  it completely.

Now, in fact,  we arrived to a main point of the discussion:
the behaviour of $m(c)$ is very different for $c>0$ and for $c\leq 0$.
As we already said, $m(c)$ for $c>0$ is topological, independent of the analytic structure of $(X,o)$. For $c\leq 0$, the knowledge of  $\Gamma$ and $F$ is not enough, in general the analytic structure of $(X,o)$
plays a crucial role. Hence, in this $c\leq 0$ case,  the  study of
the jumping numbers and multiplicities goes in parallel with the study of (yet open problem of classification of) analytic structures supported on a fixed topological type.

In subsection \ref{ss:cnegative} we  provide some concrete examples in the case $c<0$
to show how the variation on the analytic structure enters in the picture.
Moreover, for weighted homogeneous germs we determine  $m(c)$ for any $c\in\R$ in terms of the (analytic) Poincar\'e series of $(X,o)$.
But a  more  complete study regarding $\{m(c)\}_{c<0}$ will be presented in  forthcoming manuscript including several other  analytic aspects and connections.

The difference between the negative and positive half-lines is already visible via
the following formula (cf. the first part of the proof
of Proposition \ref{lem:chi}).
\begin{equation}\label{eq:sum}
\sum_{c'\leq c}m(c')=\dim\, \cO_{X,o}/ \fI(c ,F)=\mathfrak{h}(L(c))=
\chi(L(c))+p_g -h^1(\cO_{\tX}(-L(c)).
\end{equation}
Here $p_g=p_g(X,o)$ is the geometric genus  of $(X,o)$,
defined as $\dim \, H^1(\widetilde{X},\cO_{\widetilde{X}})$. Usually it is not a  topological invariant of $(X,o)$, however, it is a $c$--independent `universal'
analytic invariant.  Thus, let us focus on the
difference $\mathfrak{h}(L(c))-p_g$. It has two terms, $\chi(L(c))$ is combinatorial,
however,  $h^1(\cO_{\tX}(-L(c))$ is not. The point is that for
$c\geq 0$ this last term is zero (see the second part of the proof of Proposition \ref{lem:chi}) by the Generalized Grauert-Riemenshneider vanishing theorem, independently of the analytic structure of $(X,o)$. On the other hand, for $c<0$ this term can be nonzero, and usually depends essentially on the choice of analytic structure
(see also examples from subsection \ref{ss:cnegative}, and Remark \ref{rem:pinkham}).
We can say that $c>0$ is the `Riemann--Roch zone', while $c\leq 0$ is the
`$h^0$--zone'.

\bekezdes
On the other hand,
the smallest jumping number --- independently of the fact that it
is negative or not --- can be deduced combinatorially.

\begin{lemma}\label{lem:min}
(a) The value   $\mathfrak{c}:=
\min\{c\,:\, m(c)>0\}$ can be determined from $\Gamma$ and $F$ as follows.

Write $Z_K=\sum_ik_iE_i$ and $F=\sum _im_iE_i$.
Then
\begin{equation}\label{eq:cmin}
\mathfrak{c}=\min_{i\in\cV}  \{(1-k_i)/m_i \}.\end{equation}
Moreover, $m(\mathfrak{c})=1$.

(b) (`weak/support periodicity') \
Assume  that there exists a function $f:(X,o)\to (\bC,0)$ and $n\in \Z_{>0}$
such that ${\rm div}_{\phi,E}(f^n)=F$. Then $m(c+\frac{1}{n})\geq m(c)$ for any $c\in\R$.
\end{lemma}
In the case of $\mathfrak{c}>0$, part (a) was already known (see eg. \cite{Carrami2}), $\mathfrak{c}$ is known as the log-canonical threshold of $F$.

Note that if $(X,o)$ is rational and $\phi$ arbitrary, or $(X,o)$ is minimally elliptic and
$\phi$ is minimal, then any $F\in\calS\setminus \{0\}$ is realized as ${\rm div}_{\phi,E}(f)$ for some
 $f$ \cite[Corollary 7.1.13 and Theorem 7.2.31]{Nkonyv}.

\begin{proof} (a)
Let $\mathfrak{c}'$ be the right hand side of   (\ref{eq:cmin}).
Then $L(\mathfrak{c}')\leq E$, but at least one of its coefficients is one.
Hence  $\fI(\mathfrak{c}' ,F)=\mathfrak{m}_{X,o}$ and
 $\fI(\mathfrak{c}'-\epsilon ,F)=0$.

 (b) Note that $F/n\in L$ hence $L(c+\frac{1}{n})=L(c)+F/n$ for every $c\in\R$. Then consider the commutative diagram (with simplification $\fI(c)=\fI(c,F)$):

\[\begin{tikzcd}
	0 & {\mathfrak{I}(c)} & {\mathfrak{I}(c-\ev)} & {\mathfrak{I}(c-\ev)/\mathfrak{I}(c)} \\
	0 & {\mathfrak{I}(c+1/n)} & {\mathfrak{I}(c-\ev+1/n)} & {\mathfrak{I}(c-\ev+1/n)/\mathfrak{I}(c+1/n)}
	\arrow[from=1-1, to=1-2]
	\arrow[from=1-2, to=1-3]
	\arrow["\cdot f", from=1-2, to=2-2]
	\arrow[from=1-3, to=1-4]
	\arrow["\cdot f", from=1-3, to=2-3]
	\arrow["\cdot{\overline{f}}", from=1-4, to=2-4]
	\arrow[from=2-1, to=2-2]
	\arrow[from=2-2, to=2-3]
	\arrow[from=2-3, to=2-4]
\end{tikzcd}\]

\vspace{2mm}

Above, the vertical morphisms are induced by multiplication by $f\in\cO_{X,o}$. The first two vertical
morphisms are injective. Since the intersection  $f\fI(c-\epsilon)\cap \fI(c+\frac{1}{n})$, as
subspaces of $\fI(c-\epsilon+\frac{1}{n})$,  is $f\fI(c)$, the injectivity of $\cdot \bar{f}$ follows too.
\end{proof}

\begin{example}\label{ex:logcan}
Write $Z_K=\sum_i k_i E_i$.
Following e.g.  \cite[page 56]{kollar},
$(X,o)$ is  {\it numerically log terminal, nlt} (respectively
{\it numerically log canonical, nlc})  if $k_i< 1$ (respectively $k_i\leq 1$)  for all $i\in \cV$.
(If this happens in the minimal good resolution then it happens in any good resolution.)
The {\it nlt} singularities are exactly the quotient singularities (see also \cite[Example 6.3.33]{Nkonyv}).

By (\ref{eq:cmin}), if $(X,o)$ is {\it nlt} (respectively {\it nlc})
 then  $m(c)=0$ for all $c\leq 0$  (respectively  $c<0$) and  for any $F\in \calS\setminus \{0\}$.
 In the {\it nlc} case,
  $(X,o)$ is either (special) rational or $Z_K=E$ (cf. \cite{kollar} or
\cite[Example 6.3.33]{Nkonyv}). This latter case happens if $(X,o)$ is either cusp or simple elliptic.
If $Z_K=E$ then $\mathfrak{c}=0$ and $m(0)=1$.

In the special case when  $(X,o)$ is smooth or it is of type ADE then $Z_K\leq 0$,
hence $\mathfrak{c}>0$ again.

\end{example}

\subsection{Some examples regarding $m(c)$ with $c\leq 0$. }\label{ss:cnegative}

\begin{example}\label{ex:rat}
Negative $\mathfrak{c}$ can appear even for  $(X,o)$ rational.
Take e.g. the following rational graph, and let $F$ be determined by the arrow:
if $i(a)$ is the vertex which supports the arrow, then $F= E^*_{i(a)}$.

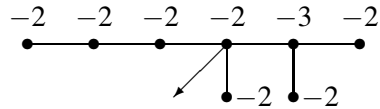
$$
    \begin{picture}(0,45)(200,0)
\put(125,25){\circle*{4}} \put(150,25){\circle*{4}}
\put(175,25){\circle*{4}} \put(200,25){\circle*{4}}
\put(200,5){\circle*{4}}  \put(225,5){\circle*{4}}
\put(225,25){\circle*{4}} \put(250,25){\circle*{4}}
 \put(125,25){\line(1,0){125}}
\put(200,25){\line(0,-1){20}}
\put(225,25){\line(0,-1){20}}
\put(200,25){\vector(-1,-1){20}}

\put(125,35){\makebox(0,0){$-2$}}
\put(150,35){\makebox(0,0){$-2$}}
\put(175,35){\makebox(0,0){$-2$}}
\put(200,35){\makebox(0,0){$-2$}}
\put(225,35){\makebox(0,0){$-3$}}
\put(250,35){\makebox(0,0){$-2$}}
\put(210,5){\makebox(0,0){$-2$}}
\put(235,5){\makebox(0,0){$-2$}}

\end{picture}
$$

Then, by a computation,   $\mathfrak{c}=-1/4$ and  $m(-1/4)=1$.

In this case $D_0\not=0$, but $m(0)=0$. Indeed, $Z_K$ has some integral coefficients
hence $D_0\not =0$. On the other hand,  both $L(0)$ and $L(-\epsilon)$ are elements of $L_{>0}$,
but they are smaller then the Artin fundamental cycle. Hence both jumping ideals are $\mathfrak{m}_{X,o}$.

\end{example}

\begin{example}\label{ex:sum} (a)
For any  $(X,o)$, for any good $\phi$ and any $F$,  the identity (\ref{eq:sum}) (via the vanishing
$h^1(\cO_{\tX}(-\lfloor Z_K\rfloor))=0$, cf. the proof of Proposition \ref{lem:chi})
in the case of $c=0$  reads as
$$\sum_{c'\leq 0}m(c')=\chi(\lfloor Z_K\rfloor) +p_g.$$
Note that  this number is independent of the choice of $F$, however, the distribution of the jumping numbers might depend on $F$.

In  rational case of Example \ref{ex:rat}, $p_g=0$ and $\chi(\lfloor Z_K\rfloor)=1$.

If $(X,o)$ is numerically Gorenstein, then $\chi(\lfloor Z_K\rfloor)=\chi(Z_K)=0$, hence $\sum_{c'\leq 0}m(c')=p_g$.

\vspace{1mm}

(b)
Consider the next rational graph, where the number of $(-2)$ curves in all
the legs is $(2n-1)$.

$$
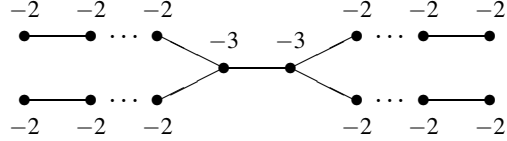
\begin{picture}(100,55)(100,-20)
\put(150,20){\makebox(0,0){\footnotesize{$-3$}}}

\put(125,-12){\makebox(0,0){\footnotesize{$-2$}}}
\put(125,32){\makebox(0,0){\footnotesize{$-2$}}}
\put(100,-12){\makebox(0,0){\footnotesize{$-2$}}}
\put(100,32){\makebox(0,0){\footnotesize{$-2$}}}
\put(75,-12){\makebox(0,0){\footnotesize{$-2$}}}
\put(75,32){\makebox(0,0){\footnotesize{$-2$}}}

\put(200,-12){\makebox(0,0){\footnotesize{$-2$}}}
\put(200,32){\makebox(0,0){\footnotesize{$-2$}}}
\put(225,-12){\makebox(0,0){\footnotesize{$-2$}}}
\put(225,32){\makebox(0,0){\footnotesize{$-2$}}}
\put(250,-12){\makebox(0,0){\footnotesize{$-2$}}}
\put(250,32){\makebox(0,0){\footnotesize{$-2$}}}

\put(175,20){\makebox(0,0){\footnotesize{$-3$}}}

\put(150,10){\circle*{4}}
\put(150,10){\line(1,0){25}}
\put(150,10){\line(-2,1){25}}
\put(150,10){\line(-2,-1){25}}
\put(175,10){\circle*{4}}
\put(175,10){\line(2,1){25}}
\put(175,10){\line(2,-1){25}}

\put(200,22){\circle*{4}}
\put(213,22){\makebox(0,0){$\dots$}}
\put(200,-2){\circle*{4}}
\put(213,-2){\makebox(0,0){$\dots$}}
\put(225,-2){\circle*{4}}
\put(225,-2){\line(1,0){25}}
\put(225,22){\circle*{4}}
\put(225,22){\line(1,0){25}}
\put(250,-2){\circle*{4}}
\put(250,22){\circle*{4}}

\put(125,22){\circle*{4}}
\put(113,22){\makebox(0,0){$\dots$}}
\put(125,-2){\circle*{4}}
\put(113,-2){\makebox(0,0){$\dots$}}

\put(100,22){\circle*{4}}
\put(75,22){\circle*{4}}
\put(100,22){\line(-1,0){25}}
\put(100,-2){\line(-1,0){25}}

\put(100,-2){\circle*{4}}
\put(75,-2){\circle*{4}}

\end{picture}
$$
Set $F:=E$, it is ${\rm div}_{\phi,E}$ of the generic function.
A computation shows that the multiplicity of $Z_K$ at the nodes is $n$,
by (\ref{eq:cmin})
$\mathfrak{c}=1-n$, and  $\chi(\lfloor Z_K\rfloor)=n$. Hence,
by part (a) $\sum _{c\leq 0}m(c)=n$  and
by Lemma \ref{lem:min}{\it (b)},
necessarily $\sum _{c\leq 0}m(c)t^c=t^{1-n}+\cdots+t^{-1}+t^0$.

\end{example}

\begin{remark}
  In the next subsection (Proposition \ref{lem:chi})
  we determine   $m(0)$  for any $(X,o)$. It depends on the choice of the analytic structure of $(X,o)$ but it does not depend on the choice of  $F$.
\end{remark}

\begin{example}\label{ex:Laufer}
Consider the following resolution graph.

    $$
    
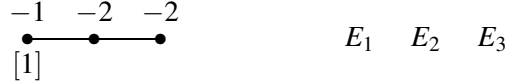
\begin{picture}(0,35)(200,5)
\put(125,25){\circle*{4}} \put(150,25){\circle*{4}}
\put(175,25){\circle*{4}}

 \put(125,25){\line(1,0){50}}

\put(125,35){\makebox(0,0){$-1$}}
\put(150,35){\makebox(0,0){$-2$}}
\put(175,35){\makebox(0,0){$-2$}}
\put(250,25){\makebox(0,0){$E_1$}}
\put(275,25){\makebox(0,0){$E_2$}}
\put(300,25){\makebox(0,0){$E_3$}}
\put(125,15){\makebox(0,0){$[1]$}}

\end{picture}
$$

The graph is weakly elliptic.
It turns out that depending on the analytic structure  $1\leq p_g\leq 3$ (see e.g.
\cite[7.2.C and 7.2.D]{Nkonyv}).
$Z_K=3E_1+2E_2+E_3$, hence $Z_K\in L$, and  $\sum_{c\leq 0}m(c)=p_g$.

Taking $E_1$ (the elliptic irreducible exceptional divisor) as a   central vertex, the link is a Seifert 3-manifold, and $(X,o)$ can be chosen weighted homogeneous.
The point is that we can choose  three different families  of weighted homogeneous
analytic structures, see \cite{pinkham} (or \cite[Example 5.1.30]{Nkonyv}).
Certain  representatives are the following.

(i) $(X,o)$ is the hypersurface singularity $z^2=y^3+x^{18}$ with $p_g=3$,

(ii)  $(X,o)$ is the hypersurface singularity $z^2=x(y^4+x^6)$ with $p_g=2$,

(iii) $(X,o)$ is not Gorenstein, with $p_g=1$.

Take $F=E_1+E_2+E_3$. Then there are three values $c\leq 0$ with non-trivial jumping divisors: $c\in\{-2,-1,0\}$.
For any analytic structure, $L(-2)=E_1$ and $L(-2-\epsilon)=0$, hence
$m(-2)=1$.  On the other hand, taking $c=0$ we get $L(0)=Z_K$, $L(-\ev)=2E_1+E_2$, and $\chi(L(0))=\chi(L(-\ev))=0$. Therefore,
as we will prove in  Proposition \ref{lem:chi}, $m(0)=\mathfrak{h}^0$,
which equals 1 if $(X,o)$ is Gorenstein (in the first two cases).
In case (iii) it should be 0 since $\sum_{c\leq 0} m(c)=1$, but
we already know that $m(-2)=1$.
Hence $\sum_{c\leq 0}m(c)t^c$ is $t^{-2}+t^{-1}+t^{0}$ in case (i),
 $t^{-2}+t^{0}$ in case (ii), and
  $t^{-2}$ in case (iii).
\end{example}

\subsection{The computation of $m(c)$ for $c\geq 0$.}\label{s:cpoz}
\bekezdes
In this
subsection we determine $m(c)$ for all $c\geq 0$, for all possible $F$'s,  and  for all analytic structures.

Though the cycles $L(c)$ are very special, still, in principle,
the output of the substitution  $c\mapsto \mathfrak{h}(L(c))$
might  depend on the analytic type of $(X,o)$, usually it cannot be determined from $\Gamma$.
This fact will follow from  Proposition \ref{lem:chi} where we  connect
$\mathfrak{h}(L(c))$  with the geometric genus $(X,o)$,
which usually is not a topological invariant of $(X,o)$.

However, for any $c>0$ we  will show that $\mathfrak{h}(L(c))-p_g $ is combinatorial. Hence, via (\ref{eq:m}),
$m(c)$ is also combinatorial,
and we will provide several formulae for it.

For  $c=0$ an additional analytic contribution should also be considered:
$$\mathfrak{h}^0:=h^0(D_0, \cO_{D_0}(K_{\tX}+\lfloor Z_K\rfloor)).$$
Recall that from the definition of $D_0$, $D_0=E$  exactly when $(X,o)$ is numerically  Gorenstein: $Z_K\in L$.
\begin{lemma}\label{lem:hnulla}
$\mathfrak{h}^0\in\{0,1\}$. More precisely,

 (a)  $\mathfrak{h}^0=0$ exactly when either
 (i) $D_0\not=E$, or (ii) $D_0=E$ but there exists $E_i$ such that
 $\cO_{\tX}(K_{\tX}+\lfloor Z_K\rfloor)|_{E_i}$ is not trivial in ${\rm Pic}^0(E_i)
 \simeq \C^{g_i}$.
 (In this second case, necessarily $g_i>0$.)

 (b)  $\mathfrak{h}^0=1$ exactly when $D_0=E$  and
     $\cO_{\tX}(K_{\tX}+\lfloor Z_K\rfloor )|_{E_i}=0$ in ${\rm Pic}^0(E_i)$ for every $i\in\cV$.
\end{lemma}

\begin{example} (a) If  $(X,o)$ is Gorenstein, then
$\cO_{\tX}(K_{\tX}+\lfloor Z_K\rfloor )=\cO_{\tX}$, and
  $\mathfrak{h}^0=h^0(\cO_E)=1$.

(b)
If $g=\sum_ig_i=0$ then  $\mathfrak{h}^0=1$ in the numerically Gorenstein case,
and it is zero otherwise.

(c) If $(X,o)$ is rational, then $g=0$ automatically, hence $\mathfrak{h}^0=1$ in the ADE (and smooth) cases (the only numerically Gorensein rational cases) and
 $\mathfrak{h}^0=0$ otherwise.

 (d) If the link is an integral homology sphere (hence $L=L'$) then  $g=0$ and
 $Z_K\in L$, hence $\mathfrak{h}^0=1$.

(e) In general $\mathfrak{h}^0$ is not topological: for the germs from
Example \ref{ex:Laufer} $\mathfrak{h}^0=1$ in the Goresntein cases and $0$ otherwise.
(Use $\Delta_0=0$, $\chi(D_0)=0$ and  the formula $m(0)=\mathfrak{h}^0$ from
Corollary \ref{cor:m0}{\it (a')}.)
\end{example}
\begin{proof}[Proof of Lemma \ref{lem:hnulla}]
    Write $\calL:=\cO_{\tX}(K_{\tX}+\lfloor Z_K\rfloor )$. Then $c_1\calL=-\{Z_K\}$.

{\it  (a)} \ Assume that $Z_K\not\in L$ (that is, $D_0\not= E$), hence $\D_0=\{Z_K\}\not=0$. Recall that $D_0$ is reduced, let  $D'$ be one of its connected components. Then
$(D', c_1\calL)<0$, since $(D',\Delta_0)\not=0$. In particular, there exists $E_{i'}$ in the support of $D'$
such that $(E_{i'},c_1\calL)<0$. Then we claim that $h^0(\cO_{D'}\otimes\calL)=0$.

Indeed, consider a sequence of reduced connected
divisors $\{D'_{n}\}_{n=1}^N$ such that   $D'_1=E_{i'}$, 
 $ D'_{m+1}=D'_m+E_{i(m)}$ with $(D'_m,E_{i(m)})>0$,  and $D'_N=D'$.
Then the vanishing follows inductively from the exact sequences
$0\to \cO_{E_{i(m)}}(-D'_m)\otimes \calL\to \cO_{D'_{m+1}}\otimes \calL\to
 \cO_{D'_{m}}\otimes \calL \to 0$.

 Assume next that $Z_K\in L$ (hence $D_0=E$) but
 $\calL|_{E_{i'}}$ is not trivial in ${\rm Pic}^0(E_{i'})$
 for some $i'$. This means that $h^0(E_{i'}, \calL|_{E_{i'}})=0$.
 Indeed, since  $\calL|_{E_{i'}}$ is topologically trivial (its Chern class is zero), any nonzero element of
  $H^0(E_{i'}, \calL|_{E_{i'}})$ would provide a trivialization of
  $\calL|_{E_{i'}}$.

  Then we can repeat the above proof  with a similar sequence with
   $D'_1=E_{i'}$ and  $D'_N=E$.

  {\it (b)} \ It follows by induction using a similar sequence of divisors
  with  $D'_N=E$ and $D'_1=E_{i'}$ arbitrary choosen. Indeed,
  $h^0(\cO_E\otimes \calL)= h^0( \cO_{D'_{m}}\otimes \calL) =h^0(\cO_{E_{i'}}\otimes \calL)= h^0(\cO_{E_{i'}})=1$.
\end{proof}

\begin{proposition}\label{lem:chi} (a) If  $c\in {\mathbb R}_{\geq 0}$ then
 $\fh(L(c))=\chi(L(c))+p_g$.

Thus, $\dim \cO_{X,o}/\fI(c,F)-p_g=\mathfrak{h}(L(c))-p_g=\chi(L(c))$
is combinatorially   computable from  $\Gamma$.

In particular, if $(X,o)$ is rational, then  $\dim \cO_{X,o}/\fI(c,F)$ is
completely combinatorial.

(b)  $m(c)=\chi(L(c))-\chi(L(c-\ev))$ for any $c>0$,
hence it is always combinatorial.

\ \ \ \ \  $m(0)=\chi(L(0))-\chi(L(-\ev))+\mathfrak{h}^0$,
hence $m^*(0):=m(0)-\mathfrak{h}^0$ is combinatorial.
\end{proposition}
The case $c>0$ in part (b) for rational $(X,o)$ was proved in \cite[Prop. 4.10]{Carrami}.

\begin{proof} (a)
Write  $L(c)= P-N$ as in Lemma \ref{lem:PN}.
Then,  by this Lemma,   $\mathfrak{h}(L(c))=\mathfrak{h}(P)$.
The short exact sequence of sheaves
$0\to \cO_{\tX}(-P)\to \cO_{\tX}\to \cO_P\to 0$ provides
the exact sequence
\[\begin{tikzcd}[column sep = small]
	0 & {H^0(\cO_{\tX})/H^0(\cO_{\tX}(-P))} & {H^0(\cO_P)} & {H^1(\cO_{\tX}(-P))} & {H^1(\cO_{\tX})} & {H^1(\cO_P)} & 0.
	\arrow[from=1-1, to=1-2]
	\arrow[from=1-2, to=1-3]
	\arrow[from=1-3, to=1-4]
	\arrow[from=1-4, to=1-5]
	\arrow[from=1-5, to=1-6]
	\arrow[from=1-6, to=1-7]
\end{tikzcd}\]
By this sequence and part {\it (b)} of Lemma \ref{lem:PN}
\begin{align*}
    \fh(P)&=\chi(P)-h^1(\cO_{\tX}(-P))+p_g \\
    &=\chi(P-N+N)-\chi(\cO_N(-P+N))-h^1(\cO_{\tX}(-P+N))+p_g \\
    &=\chi(P-N)+\chi(N)-(N,P-N)-(N,-P+N)-\chi(N)-h^1(\cO_{\tX}(-P+N))+p_g \\
    &= \chi(L(c))-h^1(\cO_{\tX}(-L(c))+p_g.
\end{align*}
Finally, we show that $h^1(\cO_{\tX}(-L(c)))=0$. To do this, set $\cL'=\cO_{\tX}(-L(c))$. Then $c_1(\cL')=-(Z_K+cF)+\D_c$.
Hence $c_1(\cL'(-K_{\tX}))=\D_c-cF$, where $F\in \calS'$. Thus  we can apply the Generalized Grauert-Riemenschneider Theorem (see \cite[Theorem 6.4.3.]{Nkonyv}),
which gives  $h^1(\tX,\cL')=0$.

Part {\it (b)} for $c>0$ follows from part {\it (a)} whenever $c>0$.
Indeed, in that case  $c-\epsilon\geq 0$ too, hence {\it (a)} applies for both $c$ and $c-\epsilon$.

Next, we assume that $c=0$. Then $L(0)=\lfloor Z_K\rfloor$. Consider the exact sequence
$0\to \cO_{\tX}(-\lfloor Z_K\rfloor) \to  \cO_{\tX}(-\lfloor Z_K\rfloor+D_0)\to
 \cO_{D_0}(-\lfloor Z_K\rfloor+D_0)\to 0$. Since by the
 Generalized Grauert-Riemenschneider Theorem (see \cite[Theorem 6.4.3.]{Nkonyv})
 $h^1( \cO_{\tX}(-\lfloor Z_K\rfloor))=0$, we get
 \begin{align*}
 m(0)&=\dim H^0( \cO_{\tX}(-\lfloor Z_K\rfloor+D_0))/
 H^0( \cO_{\tX}(-\lfloor Z_K\rfloor))=h^0 (\cO_{D_0}(-\lfloor Z_K\rfloor+D_0))\\
 &= \chi(\cO_{D_0}(-\lfloor Z_K\rfloor+D_0))+
 h^1 (\cO_{D_0}(-\lfloor Z_K\rfloor+D_0))\\
 &=\chi(L(0))-\chi(L(-\ev))+h^1 (\cO_{D_0}(-\lfloor Z_K\rfloor+D_0)).
 \end{align*}
 But $h^1 (\cO_{D_0}(-\lfloor Z_K\rfloor+D_0))= h^0 (\cO_{D_0}(K_{\tX}+\lfloor Z_K\rfloor))$ by Serre duality.
\end{proof}
\begin{corollary} \label{cor:m0}
(a) \ $m(c)=(D_c,\D_c)-\chi(D_c)-c(D_c,F)$ for $c>0$.

(a') \ $m(0)=(D_0,\D_0)-\chi(D_0)+\mathfrak{h}^0$.

(b) \ $m(c+1)=m(c)-(D_c,F)$ for $c>0$.

(b') \ $m(1)=m(0)-\mathfrak{h}^0-(D_0,F)$.

Above, in (a)--(b'), for any fixed $c\geq 0$,
the term $(D_c,F)$ equals  $\sum (E_i,F)$, where the sum is over those vertices  $i=i(a)\in \mathcal{V}$ with  $k_{i(a)}+cm_{i(a)}\in\mathbb{Z}$.

(c) \ Write $m^*(c)$ for $m(c)$ if $c>0$ and $m(0)-\mathfrak{h}^0$ if $c=0$.
Then $J_{[0,\infty)}(t):=\sum_{c\geq 0} m(c)t^c$ (where
$t$ is a formal variable) equals
$$J_{[0,\infty)}(t)=\mathfrak{h}^0t^0+
\sum_{c\in [0,1)}\ t^c\cdot  \big(m^*(c) /(1-t) -(D_c,F)t/(1-t)^2\,\big).$$
\end{corollary}
\begin{proof}
We prove {\it (a)}, the others follow similarly or  from this one.
\begin{align*}
    m(c)&=\chi(L(c-\epsilon)+D_c)-\chi(L(c-\epsilon))
    =\chi(D_c)-(D_c,L(c-\epsilon)) \\ &= \chi(D_c)-(D_c,Z_K+cF-\D_c-D_c)
    =\chi(D_c)-2\chi(D_c)+(D_c,\D_c-cF).
\end{align*}

\vspace*{-7mm}

\end{proof}

\begin{example}\label{ex:rat2} For the rational graph from Example \ref{ex:rat} one gets
 $$J_{[0,\infty)}(t)=\frac{t}{(1-t)^2}(1+t^{1/4}+t^{2/4}+t^{3/4})+\frac{2}{1-t}t^{3/4}
=2t^{3/4}+t+t^{5/4}+t^{3/2}+3t^{7/4}+\cdots$$
\end{example}
An additional interesting connection is realized by the following sum.
\begin{proposition}
    $$\sum_{c\in(0,1] } m(c)= \chi(-F)+(\{Z_K\}, F).$$
\end{proposition}
\begin{proof}
 By (\ref{eq:sum}), the left hand side is $\chi(L(1))-\chi(L(0))=\chi(\lfloor Z_K\rfloor+F)-\chi(\lfloor Z_K\rfloor)$.
\end{proof}

\begin{corollary}\label{cor:delta}
Assume that there exists a reduced Cartier divisor $(C,o)\subset (X,o)$, cut out by the function $f:(X,o)\to ({\mathbb C},0)$,  such that ${\rm div}_{\phi,E}(f)=F$, cf. \ref{bek:1.1.2}.
Then
$$\sum_{c\in(0,1] } m(c)= \delta(C,o)+(\{Z_K\}, F),$$
where $\delta(C,o)$ is the delta invariant of the abstract analytic curve germ $(C,o)$.
\end{corollary}
\begin{proof}
For the identity $\delta(C,o)=\chi(-F)$  see \cite[Example 6.3.8]{Nkonyv}.
\end{proof}
For a generalization of the formula $\chi(-F)=\delta(C,o)$ when $(X,o)$ is rational and $(C,o)$ not necessarily Cartier, see \cite{CLMN1,CLMN2}
(when an additional correction term appears).

\subsection{The weighted homogenous case}\label{ss:we}\

Our goal is to determine $\sum_{c\in\R}m(c)t^c$ completely and to discuss  some examples.

\bekezdes The affine algebraic (complex) variety $Y\subset \C^n$ is called weighted homogeneous with integral weights $\{w_i\}_{i=1}^n$
if it is stable with respect to the action
of $\C^*$ on $\C^n$, $t*(x_1,\ldots, x_n)=(t^{w_1}x_1,\ldots, t^{w_n}x_n)$.
We assume that (i) each $w_i$ is positive (that is, the action is {\it good}:
the origin is contained in the closure of any orbit), and
(ii) ${\rm gcd}\{w_i\}_i =1$ (that is, the action is {\it effective}: if $t*x=x$ for all $x\in Y$ then $t=1$).
For such a variety $Y$ the affine coordinate ring is $\Z_{\geq 0}$--graded, $R=\oplus _{\ell\geq 0}R_{\ell}$. (Usually the normality of $R$ is not guaranteed.)

A local normal surface singularity $(X,o)$ is called weighted homogeneous if there exists a normal affine surface $Y$ with a good and effective $\C^*$ action, and a singular point $y\in Y$ such that $(X,o)$ and $(Y,y)$ are analytically isomorphic as complex analytic germs. In this case, the local algebra of $(X,o)$ is graded:
we use the same notation   $\cO_{X,o}=\oplus _{\ell\geq 0}R_{\ell}$.

We fix such a singularity and let $\phi$ be its minimal good resolution. The graph is
star-shaped with $\nu$ legs (strings).
Let  $-b_0$ and  $g$ denote the  Euler decoration and the genus of
the central vertex $v_0$. Each  leg for $1\leq j \leq \nu$ has Euler decorations
$-b_{j1}, \ldots, -b_{js_j}$, $b_{ji}\geq 2$, determined by the Hirzebruch continued fraction $\alpha_i/\omega_i=[[b_{j1}, \ldots, b_{js_j}]]$, where
${\rm gcd}(\alpha_j, \omega_j)=1$, $0<\omega_j<\alpha_j$. In the graph,
for each $j$ the central vertex $v_0$ is connected to $v_{j1}$
by an edge. We assume that either $g>0$ or $\nu\geq 3$.
(Otherwise, the graph is a string with all $g_i=0$, hence $(X,o)$ is a quotient singularity, or a  numerically log terminal singularity with all $k_i<1$; this case
is treated in Example \ref{ex:logcan} and in subsection \ref{s:cpoz}.)

The link is a Seifert 3-manifold with Seifert invariants $(b_0,g; \{(\alpha_j,\omega_j)\}_j)$. It has several important numerical invariants (for details see e.g. \cite{pinkham} and \cite{Nkonyv}):

$\bullet$ \  the orbifold Euler number $e:= -b_0+\sum_j \omega_j/\alpha_j$ ($e<0$ by the negative definiteness of $(\,,\,)$)

$\bullet$ \  the orbifold Euler characteristic $\chi:= 2-2g -\sum_j
(\alpha_j-1)/\alpha_j$,

$\bullet$ \  $\mathfrak{r}=\chi/e$, where $-\mathfrak{r}$ is called the log discrepancy of $(X,o)$,

$\bullet$ \  the order of $H=L'/L={\rm Tors}(H_1(M,\Z))$, it equals
$|H|=|e|\cdot \prod_j\alpha_j$,

$\bullet$ \  the order of $[E^*_{v_0}]$ (of the generic $S^1$ orbit) in $L'/L$, it equals $\mathfrak{o}=|e|\alpha$,
where $\alpha:={\rm lcm}_j\{\alpha_j\}$,

$\bullet$ \  $k_0$, the $E_{v_0}$--coefficient of $Z_K$, it equals $k_0=1+
\mathfrak{r}$,

$\bullet$ \  the $E_{v_0}$--coefficient of $E^*_{v_0}$, it equals
$(\prod_j\alpha_j)/|H|$.

Associated with  $\phi$ we choose $F:= \mathfrak{o}E^*_{v_0}$.
The rational cycle
$E^*_{v_0}$ is associated with one arrow attached to the central vertex,
and $\mathfrak{o}$ is the smallest integer such that $\mathfrak{o} E^*_{v_0}$ is integral.

Note that
$m_{v_0}$,   the  $E_{v_0}$--multiplicity of $F$,  is
$\mathfrak{o}\cdot \prod_j\alpha_j/|H|=\alpha$.

For any cycle $l$ we denote by  $l_0$ its  $E_{v_0}$--coefficient.

\begin{theorem} \label{th:Weighted}
    Let $(X,o)$ and $F$ be as before. Then

    (a)  If \ $\fI(c,F)\subsetneq\fI(c-\ev,F)$ then $k_0+cm_{v_0} \in\Z$, and
      $\fI(c-\ev,F)/\fI(c,F))\simeq  R_{L(c)_0-1}$,

(b)
    $$\sum_{c\in \R}m(c)t^c = \sum_{\ell\geq 0}
  \  \dim(R_{\ell})\, t^{  (\ell+1-k_0)/\alpha}=
  \sum_{\ell\geq 0}
  \  \dim(R_{\ell})\, t^{  (\ell-\mathfrak{r})/\alpha}
  .
 $$
 Thus, $\sum_cm(c)t^c $ follows by a simple way
 from the Poincar\'e series $P(t):= \sum_{\ell\geq 0} \, \dim(R_\ell)\, t^\ell$ of
 $(X,o)$.
\end{theorem}
Theorem \ref{th:Weighted} follows from the next technical result.

\begin{proposition}\label{prop:tech} For any $c\in\R$ and $s\in\calS$,
 $L(c)\leq s$ if and only if $L(c)_0\leq s_0$.
\end{proposition}

\begin{proof}[Proof of Theorem \ref{th:Weighted} via Proposition \ref{prop:tech}]
  We consider the filtration $\{\calF (\ell)\}_{\ell\geq 0}$ of $\cO_{X,o}$
   associated with the divisor $E_{v_0}$, i.e. $\calF(\ell):= \{f\in \cO_{X,o}\,:\, {\rm div}_{\phi,E}(f)\geq \ell E_{v_0}\,\}$.

   We claim that $\fI(c,F)=\calF(L(c)_0)$.

   Indeed, if $f\in\fI(c,F)=\phi_*\cO_{\tX}(-L(c))$ then ${\rm div}_{\phi,E}(f)
   \geq L(c)\geq L(c)_0E_{v_0}$. Also,  if $f\in\calF(L(c)_0)$ then
${\rm div}_{\phi,E}(f)\in\calS$ and ${\rm div}_{\phi,E}(f)\geq L(c)_0E_{v_0}$, hence
$f\in\phi_*\cO_{\tX}(-L(c))$  too by the Proposition.

But, for  homogeneous singularities the $\{\calF(\ell)\}_\ell$
filtration coincides with the filtration induced by the $\C^*$--action:
$\calF(\ell)=\oplus_{\ell'\geq \ell}R_{\ell'}$, see e.g.
\cite[Example 8.6.13]{Nkonyv}.
Therefore, $\fI(c,F)=\oplus_{\ell\geq L(c)_0}R_{\ell}$.
 If \ $\fI(c,F)\subsetneq\fI(c-\ev,F)$ then necessarily  $k_0+cm_{v_0} \in\Z$, and
      $\fI(c-\ev,F)=\oplus_{\ell\geq L(c)_0-1}R_{\ell}$.
\end{proof}

\begin{proof}[Proof of Proposition \ref{prop:tech}]
Fix $s\in \calS$ and assume that $s_0\geq L(c)_0$. Then $(Z_K+cF-E,E_v)=0$ for any vertex except for the node and the end vertices, and it equals
 1  for any end vertex. Therefore, $(s-Z_K-cF+E, E_v)\leq 0$  for any
 $v\not=v_0$. On the other hand, $(s-Z_K-cF+E)_0=s_0-(Z_K+cF)_0+1\geq r$,
 where $r:= (\Delta_c)_0+1>0$.

 Recall that $(E^*_{v_0})_0=\prod_j\alpha_j/|H|$, and  set
 $r':=(r|H|)/\prod_j\alpha_j>0$. Then $(r'E^*_{v_0})_0=r$. This shows that
 the cycle
 $$C:= s-Z_K-cF+E-r'E^*_{v_0}$$
 is supported on the legs, and $(C,E_v)\leq 0$ for any vertex $v$ of the legs.
 Thus, $C$ is in the Lipman cone of the legs, hence
 $C\geq 0$. This reads as
 $$s\geq Z_K+cF-E+r'E^*_{v_0}\geq L(c)-E+r'E^*_{v_0}.$$
 Finally  use the facts that all the coefficients of $E^*_{v_0}$ are strict positive and
 $s$ is an integral cycle.
\end{proof}
\begin{example} (i)
Assume that $M$ is an integral homology sphere. Then it is the Seifert 3-manifold
$\Sigma(\alpha_1,\ldots, \alpha_\nu)$. Then, by \cite{EN,Neumann} $(X,o)$ is a Brieskorn-Hamm complete intersection, the $\{\alpha_j\}_j$'s are pairwise relative prime, $\alpha=\prod_j\alpha_j$, and
its Poincar\'e series is
$$P(t)=\frac{(1-t^\alpha )^{\nu-2}}{\prod_j (1-t^{\alpha/\alpha_j})}.$$
(ii) If $M$ is a rational homology sphere (equivalently, $g=0$)  then again,  the Poincar\'e series
can be determined from the Seifert invariants by the Dolgachev-Pinkham-Demazure formula \cite{pinkham}
$$P(t)=\sum_{\ell\geq 0}\, \max\{0, 1+\ell b_0-\sum_j \lceil \ell\omega_j/\alpha_j\rceil \}\cdot t^\ell. $$
(iii) If $(X,o)$ is a hyperurface singularity given by an isolated quasihomogeneous germ
with weights  $(w_1,w_2,w_3)$ and degree $d$ then
$$P(t)=\frac{1-t^d}{\prod_i(1-t^{w_i})}.$$
\end{example}

\begin{example}\label{ex:Laufer2}  (a)
    Consider the weighted homogeneous hypersurface germs (i) and (ii) from Example \ref{ex:Laufer}.
In that Example
we considered  $F=E=E^*_3$, which does not apply to the discussion from this subsection \ref{ss:we}.
Hence, let us take at this time $F=E_{v_o}^*=Z_K$. Then by  a very similar  computation
$\sum_{c\leq 0}m_{(i)}(c)t^c=t^{-2/3}+t^{-1/3}+t^{0}$ in case (i),
and  $\sum_{c\leq 0}m_{(ii)}(c)t^c=t^{-2/3}+t^{0}$ in case (ii). On the other hand, for $c>0$
(by subsection \ref{s:cpoz})
the values $m(c)$ should agree (since the germs  share the same  graph).
Hence $\sum_{c}m_{(i)}(c)t^c-\sum_{c}m_{(ii)}(c)t^c =t^{-1/3}$.

Let us compare this fact with the statement of Theorem \ref{th:Weighted}.
In case (i), $z^2=y^3+x^{18}$ has weights (1,6,9) and degree 18.
    In case (ii), $z^2=x(y^4+x^6)$, the weights are (2,3,7) and  the degree is 14.
    Hence the  Poincar\'e series are
    $$P_{(i)}(t)=\frac{1-t^{18}}{(1-t)(1-t^6)(1-t^9)} \ \ \ \mbox{and} \ \ \
P_{(ii)}(t)= \frac{1-t^{14}}{(1-t^2)(1-t^3)(1-t^7)}.$$
Since the `substitution' is $t^\ell \mapsto t^{(\ell-2)/3}$,  this is indeed
in agreement with the above computation since (rather interestingly) $P_{(i)}(t)-P_{(ii)}(t)=t$ (and
the above substitution realizes $t\mapsto  t^{-1/3}$).

(b) The above computation, valid for $F:=E_{v_0}^*=E_1^*=Z_K $ can be compared with the computation from
 Example \ref{ex:Laufer}. In fact, let us take $F':=3E_3^*=3E$. Then the multiplicities in the cycles
 $F=Z_K$ and $F'=3E$ of the central vertex agree. Then, surprisingly, the computation from
  Example \ref{ex:Laufer} and from above show that the expression $\sum_{c\leq 0}m(c)t^c$,
  computed individually for any choice (i)--(iii) of the analytic type, is the same for $F$ and $F'$
  (though $F'$ is not a rational multiple of $E_{v_0}^*$).

This might suggests that
if we drop the assumption that $F$ is a rational multiple of $E^*_{v_0}$, then
in the computation of $\sum_{c}m(c)t^c$, associated with $F$,
 only the
the central vertex multiplicity of $F$ counts. Well,
 for $c>0$ this is not the case. For $c>0$, though  $\sum_{c>o} m(c)t^c$
  is independent of the choice of the analytic type (i)---(iii), it turns out that it depends on the choice of
  $F$ and $F'$:
  $$\sum_{c>0}m(c)t^c=\left\{\begin{array}{ll}
 \sum_{n>0}\left\lceil \frac{n}{3} \right\rceil t^{n/3} & \mbox{for $F$}, \\
  \sum_{n>0}nt^{n/3} & \mbox{for $F'$}.
  \end{array}\right. $$

\end{example}

\begin{remark}\label{rem:pinkham} (a)
Let us consider again the formula from Theorem \ref{th:Weighted}{\it (b)}, which gives the identification
$m((\ell-\mathfrak{r})/\alpha)=\dim(R_\ell)$.  Though each homogeneous component  $R_{\ell}$ of the local algebra of $(X,o)$ a priori is analytical, in subsection \ref{s:cpoz} proved that
$m(c)$ for $c>0$ is topological. Hence, via our formula $\dim(R_{\ell})$ is topological for any $\ell>\mathfrak{r}$. This can be seen in the theory of weighted homogeneous singularities as well.
Indeed, for  $\ell>\mathfrak{r}$,  the $h^1$ of a certain line bundle is vanishing, whose $h^0$ is exactly $R_{\ell}$
(for details see e.g. \cite{pinkham} or \cite[Remark 5.1.29]{Nkonyv}).

(b) On the other hand, the formula $(\dag)$ $\sum _{c\leq 0}m(c)=\chi(\lfloor Z_K\rfloor) +p_g$
from Example \ref{ex:sum} at this generality
is new even from the point of view of the Dolgachev-Pinkham-Demazure theory of
weighted homogeneous singularities (reinterpreted via Theorem \ref{th:Weighted}{\it (b)}).
Indeed, in the Gorenstein case one has  $\chi(\lfloor Z_K\rfloor)=0$, and  it is known  that
by duality $\sum _{\ell\leq\mathfrak{r}}\dim(R_\ell)t^\ell$ corresponds to a finite Poincar\'e
series of $H^1(\tX,\cO_{\tX})$  (graded by  the $\C^*$--action), see e.g. \cite[Example 6.3.30 {\it (d)}]{Nkonyv}.
Hence, $\sum _{\ell\leq \mathfrak{r}} \dim R_\ell=p_g$.
However, for non-Gorenstein germ, the `broken duality'
formula $(\dag)$, even at this numerical level $\sum _{\ell\leq \mathfrak{r}} \dim R_\ell=p_g +\chi(\lfloor Z_K\rfloor) $  was  unknown (at least for the authors).
\end{remark}

 \section{Jumping multiplicities from the multiplicities of $F$}

\subsection{The definition of the multiplicities $\{m_i\}_{i\in\mathcal{W}}$}

\bekezdes\label{bek:multext}
 Let us fix the resolution graph $\Gamma$ with vertices $\mathcal{V}$ and the cycle $F\in\calS\setminus\{0\}$
as above. We write $F=\sum_{i\in\mathcal V} m_iE_i$. We complete $\Gamma$ and its decorations (self-intersection numbers $e_i=E_i^2$, genera $g_i$  and multiplicities
$m_i$, all indexed by $\mathcal{V}$) to a `virtual embedded resolution graph'.
We complete the vertices of $\Gamma$ with $-(E,F)$ arrowhead vertices, each
vertex $i\in\mathcal{V}$ will support $-(E_i,F)$ arrows, and each of them
will be decorated by multiplicity 1. Let $\mathcal{A}$ denote the set of
arrows, $\mathcal{W}:=\mathcal{V}\cup \mathcal {A}$. Hence each element of $\mathcal{W}$
has a multiplicity, $\{m_i\}_{i\in\mathcal{V}}$ is given by $F$, all others are $m_a=1$ for
all $a\in\mathcal{A}$. (If we fix the topological type of $(X,o)$ by $\Gamma$, then for a convenient analytic type of $(X,o)$ there exists a function $f:(X,o)\to (\mathbb{C},0)$
such that ${\rm div}_{\phi,E}(f)=F$ and the embedded resolution graph associated with $\phi$ and $f$ is this completed graph with arrows and multiplicities
$\{m_v\}_{v\in \mathcal {W}}$.)

For any $v\in \mathcal{V}$ let $\mathcal{W}_v$ be the set of all (usual and arrowhead)
vertices adjacent with $v$, and $\mathcal{V}_v=\mathcal{W}_v\cap \mathcal{V}$,
   $\mathcal{A}_v=\mathcal{W}_v\cap \mathcal{A}$.

Finally, we consider  certain  rational numbers $\{\beta_v\}_{v\in \mathcal{V}}$,
 $\beta_v\in[0,1)$. It is convenient to set also $\beta_a:=0$ for any
$a\in\mathcal{A}$. (For a geometric interpretation see Theorem \ref{th:main_mult} and subsection \ref{ss:MRHS}.)

\subsection{The definition of $TSp_{[0,1)}$}

\bekezdes
Next we define a combinatorial invariant $TSp_{[0,1)}=TSp_{[0,1)}(\{\beta_v\}_{v\in\cV})$
associated with the completed $\Gamma$, multiplicity system $\{m_v\}_{v\in \mathcal{W}}$ and rational numbers
$\{\beta_v\}_{v\in\mathcal{V}}$ ($\beta_v\in[0,1)$). Motivated by the connections from section \ref{s:Hodge},
we will call it `topological spectrum'.

For any  $v\in \cV$   and $s\in \{0, 1, \ldots , m_v-1\}$
we set
    \begin{align*}
     R^s_v(\beta) &=\sum_{w\in \cW_v} \{ -\beta_w+\frac{m_w}{m_v}(s+\beta_v) \}
     \ \text{and }\alpha^s_v =\frac{s+\beta_v}{m_v}.
    \end{align*}

Then we define
$$ TSp_{[0,1),v}^\beta=\sum_{s=0}^{m_v-1}(R^s_v(\beta)+g_v-1)t^{\alpha^s_v}.$$
Let $\cE$ be the set of edges of $\Gamma$. If $e\in\cE$ connects $v,w\in\cV$, then
we set   $m_e:={\rm gcd}(m_v,m_w)$.
Also, we  consider the following system of equations in $\gamma_e$
\begin{equation}\label{cases}\begin{cases}
\beta_v=\{m_v\gamma_e/m_e\} \\
\beta_w=\{m_w\gamma_e/m_e\}.
\end{cases}\end{equation}
Either it has a solution $\gamma_e\in [0,1)$ or it does not.
If it has a solution then that solution is unique in  $[0,1)$.
  We also introduce
    $$\gamma_e^s=\begin{cases}
        (\gamma_e+s)/m_e \ &\text{if } \gamma_e\in [0,1) \text{ exists.} \\
        -\infty \ &\text{otherwise.}
    \end{cases}$$
    and define $TSp_{[0,1),e}^\beta=\sum_{s=0}^{m_e-1}t^{\gamma_e^s}$, where $t^{-\infty}=0$.

Finally let
$$TSp_{[0,1)}^\beta:=\sum_{v\in \cV}TSp_{[0,1),v}^\beta+\sum_{e\in \cE}TSp_{[0,1),e}^\beta.$$

\begin{theorem} \label{th:main_mult}
Fix the resolution with graph $\Gamma$ and $F\in \calS\setminus \{0\}$.
Consider the complete graph with multiplicities $\{m_w\}_{w\in\cW}$ and
$\{\beta_w\}_{w\in \mathcal{W}}$  defined as
$\beta_v:=\{-k_v\}$ for $v\in \cV$ and $\beta_a=0$ for $a\in\cA$.

Then the jumping multiplicities associated with jumping numbers $c\in[0,1)$ are
determined as follows:
    $$\sum_{c\in [0,1)}m^*(c)t^c=-\mathfrak{h}^0t^0+ \sum_{c\in [0,1)}m(c)t^c=TSp_{[0,1)}^\beta.$$
\end{theorem}
In the proof we will need the following lemma.

\begin{lemma} \label{lem:system}
For a fixed $c\in [0,1)$ and an edge $e\in \cE$ which connects $v$ and $w$,
the following two facts are equivalent.

(1) The system (\ref{cases}) admits a solution $\gamma_e\in[0,1)$ and there exists
$s_e\in\{0, 1,\dots , m_e-1\}$ such that $(\gamma_e+s_e)/m_e=c$,

(2) there exist
$s_v\in\{0, 1,\dots , m_v-1\}$   and $s_w\in\{0, 1,\dots , m_w-1\}$ such that $(\beta_v+s_v)/m_v=(\beta_w+s_w)/m_w=c$.

The pairs $(\gamma_e,s_e)$ and $(s_v,s_w)$ determine each other in a unique way.
\end{lemma}
\begin{proof}
    If {\it (1)} holds then
    $s_v:=\lfloor (\gamma_e+s_e)m_v/m_e\rfloor $ and   $s_w:=\lfloor (\gamma_e+s_e)m_w/m_e\rfloor $ satisfy {\it (2)}. On the other hand, if {\it (2)} is satisfied by $(s_v,s_w)$, then $\beta_v=\{cm_v\}$ and $\beta_w=\{cm_w\}$.
    Hence if we write $cm_e=\gamma_e+s_e$ (i.e. $\gamma_e:=\{cm_e\}$) then
    $\gamma_e\in[0,1)$ is a solution of  the system  and $(\gamma_e+s_e)/m_e=c$ too.
\end{proof}

\begin{proof}[Proof of Theorem \ref{th:main_mult}]
Using the definitions of $\{\beta_w\}_{w\in \cW}$ we have
$\alpha^s_v=\frac{s+\{-k_v\}}{m_v}$, and $R^s_v(\beta)=\sum_{w\in \cW_v}\{k_w+\frac{m_w}{m_v}(s+\{-k_v\})\}$.
   In the proof below we just write $R^s_v$ for $R^s_v(\beta)$.

    Let us rewrite $TSp_{[0,1)}^\beta$ as
  $ \sum_{c\in [0,1)}\tau(c)\,t^c$.
    We want to show that $\tau(c)=m^*(c)$.

    For a fixed $c\in [0,1)$ we have
    \begin{equation}
   \tau(c)=\sum_{v\in \cV, \exists s \ \alpha^s_v=c} (R^s_v+g_v-1)+\sum_{e\in \cE, \exists s \ \gamma^s_e=c } 1.  \end{equation}
Then we  apply  Lemma \ref{lem:system} and the fact that
$-\beta_v+cm_v\in{\mathbb Z}$ if and only if $v$ is in the support of $D_c$
(and we use the notation $\cV(D_c)$ and $\cE(D_c)$ for the vertices and edges identified by the support of $D_c$). Thus,
    \begin{align}
        \tau(c) &=\sum_{v\in \cV, \exists s \ \alpha^s_v=c} (R^s_v+g_v-1)+\sum_{e\in \cE, \exists s \ \gamma^s_e=c } 1 \\
        &= \sum_{v\in \cV(D_c)} (R^{cm_v-\{-k_v\}}_v+g_v-1) + \#\{e\in \cE(D_c)\} \\
        &=  \sum_{v\in \cV(D_c)} R^{cm_v-\{-k_v\}}_v +\sum_{v\in \cV(D_c)} g_v + \#\{e\in \cE(D_c)\} - \#\{v\in \cV( D_c)\} \\
        &= \sum_{v\in \cV(D_c)} R^{cm_v-\{-k_v\}}_v
        -\chi(D_c).
    \end{align}
    Furthermore, for any $v\in \cV(D_c)$,
    \begin{align}
        R^{cm_v-\{-k_v\}}_v &=\sum_{w\in \cW_v} \{k_w+\frac{m_w}{m_v}(cm_v-\{-k_v\}+\{-k_v\})\} \\
       &=\sum_{u\in \cV_v} \{ k_u+m_u c \}+ \
       \sum_{a\in \cA_v} c \\
        &=(\{Z_K+cF\}, E_v)- (cF,E_v).
    \end{align}
Hence, via  Corollary \ref{cor:m0}{\it (a)}
    \begin{align}
        \tau(c)&=
       \big( \sum_{v\in \cV(D_c)}(\Delta_c, E_v)- (cF,E_v)\, \big)
        -\chi(D_c)\\
            &=(\D_c,D_c)-(cF,D_c)-\chi(D_c))=m^*(c)
    \end{align}

\vspace*{-9mm}

\end{proof}

\section{Jumping multiplicities and Hodge spectrum}\label{s:Hodge}

\subsection{The Milnor package and Hodge spectrum of an analytic germ}\label{ss:Milnor}

\bekezdes Let $(X,o)$ be a normal surface singularity  and $f:(X,o)\to (\mathbb {C},0)$
an analytic germ which defines an isolated singularity $(C,o):=(f^{-1}(0),o)\subset (X,o)$.
 Let $F_o$ be the Milnor fiber of $f$. It is a 1-dimensional Stein space, hence its
 cohomology $H^n:=H^n(F_o,\C)$ can be  nonzero only for $n=0,1$. Moreover, $F_o$ is connected, hence
 $H^0=\C$. Let $T^*$ be the monodromy operator $T^*:H^*\to H^*$.
 In fact, $T^0={\rm id}_{\C}$. The operator $T^n$ defines a generalized eigenspace decomposition
 $H^n=\oplus_{\lambda}H^n_\lambda$. The corresponding dimensions are codified in the monodromy zeta function
 $$\zeta(t)=\prod_{\lambda} (t-\lambda)^{\dim\, H^1_\lambda-\dim\, H^0_\lambda}=
 \det(t\cdot Id-T^1)/\det(t\cdot Id-T^0). $$
 If $\phi$ is a log embedded resolution of $(C,o)\subset (X,o)$ (and we use the notations of
 \ref{bek:multext})
 then by A'Campo's Theorem \cite{AC}
 $$\zeta(t)=\prod _{v\in\cV} (1-t^{m_v})^{2g_v-2+\#(\cW_v)}.$$
In particular, $\dim H^0_1=1$ and $\dim H^1_1=\#\cA+2g+2h-1$.

By \cite{steenbrink,saito,saito2} $H^1$ carries a mixed Hodge structure, set $H^{pq}:=Gr^W_{p+q}Gr^p_F H^1$.
The semisimple part of $T^1$ acts on each $H^{pq}$, let $H^{pq}_\lambda$ be the eigenspace with dimension
$h^{pq}_\lambda$.
Then the spectrum of $H^1$ is defined as $Sp(H^1):=\sum_{p,q, \lfloor \alpha\rfloor=p}h^{pq}_{e^{2\pi i\alpha}}t^\alpha $.
One defines similarly the spectrum $Sp(H^0)$ of $H^0$ too, but in fact it is $t^0$.
Hence, the `total' (Euler characterisitc type)  spectrum of $f$, as lifting of $\zeta$, is
$$Sp(f,t):=Sp(H^1)-Sp(H^0)=-t^0+\sum_{p,q, \lfloor\alpha\rfloor=p}\ h^{pq}_{e^{2\pi i\alpha}}\ t^\alpha\in \Z[\Q].$$
We write $Sp(f,t)=\sum _{\alpha}\sigma(\alpha)t^\alpha.$
It turns out that $\sigma(\alpha )=0$ whenever $\alpha\not\in[0,2)$,
$\sigma(0)=g+h-1$ and $\sigma(1)=\#\cA+g+h-1$.  Moreover, for any $\alpha\in(0,1)$,  one has the symmetry
$\sigma(\alpha)=\sigma(2-\alpha)$. Hence $Sp(f,t)$ is completely determined by
$Sp_{[0,1]}(f,t):=\sum _{\alpha\in[0,1]}\sigma(\alpha)t^\alpha $.

In the next theorem we fix an embedded resolution and its decorated graph (as in \ref{bek:multext}).

\begin{theorem}\label{th:NS}\cite{SSS,NS}
 $$Sp_{[0,1]}(f,t)=(g+h-1)t^0+(\#\cA+g+h-1)t^1 + \sum_{ e\in \cE} \
 \sum_{0<s<m_e}t^{\frac{s}{m_e}}+\sum_{v\in \cV} \
 \sum_{0<s<m_v}(R^s_v+g_v-1)t^{\frac{s}{m_v}},$$
where $R^s_v:= \sum_{w\in \cW_v}\{sm_w/m_v\}$.
\end{theorem}

Note that the right hand side is exactly $TSp_{[0,1]}^{\beta=0}$.
Hence, Theorems \ref{th:main_mult} and \ref{th:NS} combined give
\begin{corollary}\label{cor:NS} Assume that $(X,o)$ is a numerically Gorenstein singularity, and
$f:(X,o)\to (\C,0)$ is an analytic germ with $(f^{-1}(0),o)$ isolated singularity.
Fix an embedded  good resolution $\phi$ and set $F:={\rm div}_{\phi,E}(f)\in\calS\setminus \{0\}$.
Then
    \begin{equation}\label{eq:extb}
       \sum_{c\in [0,1]}m^*(c)t^c=-\mathfrak{h}^0t^0+ \sum_{c\in [0,1]}m(c)t^c=TSp_{[0,1]}^{\beta=0}=
       Sp_{[0,1]}(f,t).
\end{equation}
\end{corollary}
This identity was proved in case $(X,o)=(\C^2,0)$ by Budur \cite{budur}.

Though both invariants $J_{[0,1]}(t)$ and $Sp_{[0,1]}(f,t)$ are a priori analytical, and they might depend on the analytic structure of $(X,o)$, it turns out that this is not the case, and both of them equal the combinatorial expression $TSp_{[0,1]}^{\beta=0}$.

Recall that if $(X,o)$ is smooth then in the above formulae
$m^*(0)=-1$, $\mathfrak{h}^0=1$, $m(0)=0$, $m(1)=\#\cA-1$ and  $g=h=0$. The $-t^0$ term  corresponds to
$Sp(H^0)$.

\bekezdes
If $(X,o)$ is not numerically Gorenstein then the identity
 $\sum_{c\in [0,1]}m^*(c)t^c= Sp_{[0,1]}(f,t)$ is not true anymore. It fails already for rational
 germs $(X,o)$.

\begin{example}\label{ex:rat3}
Consider the rational graph $\Gamma$  as in Example \ref{ex:rat}. Assume that $f$ is a germ whose strict transform corresponds to the arrow of the graph, and $F={\rm div}_{\phi,E}(f)$. Then
$J_{[0,1]}(t)= 2t^{3/4}+t$, cf. \ref{ex:rat2}. But $Sp_{[0,1]}(f,t)=-t^0+t^{1/4}+t^{2/4}+t^{3/4}$, hence they are different.
\end{example}

The above example shows that we have to modify the above spectrum expression
(which equals $TSp_{[0,1]}^{\beta=0}$)
if we wish to recover
$TSp_{[0,1]}^{\beta}$, where $\beta$ is determined by $\{Z_K\}$ as in Theorem \ref{th:main_mult}.
This is the subject of the next subsections. We have to replace the cohomology of the Milnor fiber
$H^*(F_o,\C)$ by a cohomology with coefficient in a mixed Hodge module. That is, we will replace the
constant pure  Hodge module  with a different Hodge module of $(X,o)$.
The needed mixed Hodge module will be determined by a variation of Hodge structures on $X\setminus \{0\}$, in fact, by a finite order character of $\pi_1(X\setminus\{o\})$.
This can be identified in many different ways, one conceptual way is via a (regular) covering
of $X\setminus \{o\}$. If the link is a rational homology sphere then we can rely on the topological canonical covering, if $(X,o)$ is $\Q$--Gorenstein then we can rely on the
(analytic) canonical covering (see the constructions below).

\subsection{Mixed Hodge modules on $(X,o)$}\label{ss:MHM}
\bekezdes Let $MHM(X,o)$ be the category of {\it mixed Hodge modules}  on $(X,o)$ \cite{saito}.
If $f:(X,o)\to (\C,0)$ defines an isolated singularity $(C,o)\subset (X,o)$ then
there exists a nearby cycle functor $\psi_f:MHM(X,o)\to MHM(C,o)$ \cite{saito}.
It has a functor automorphism $T_s$ of finite order (the semisimple part of the monodromy action).
The nearby cycle functor  is compatible with the corresponding functor at the level of perverse shaves
\cite{deligne} by the forgetful functor $rat:MHM(X,o)\to
Perv(X,o)$.

A mixed Hodge module ${\mathcal M}$ is called {\it smooth} if $rat(\mathcal{M}) $ is a local system \cite{saito}. A module  $\mathcal {M}$ is called  {\it pure  of weight $n$} (or polarized Hodge module of weight $n$) if $Gr^W_i\mathcal{M}=0$ for $i\not=n$. The category of smooth polarized Hodge modules is equivalent with the category of (admissible) variations of polarized  Hodge structures.

The relation with the Milnor fibers from the previous subsection \ref{ss:Milnor} is the following.
Let $i_o:\{o\}\to (C,o)$ be the inclusion. Then $H^*(i_o^*\psi_f \C_{X,o})=H^*(F_o,\C)$.

The point is that if we replace the constant Hodge module by an arbitrary ${\mathcal M}\in MHM(X,o)$,
then each $H^k(i_o^*\psi_f\,{\mathcal M})$ carries a mixed Hodge structure (together with a finite monodromy action), hence we can define its spectrum $Sp(H^k(i_o^*\psi_f\,{\mathcal M}))$ similarly as we defined
$Sp(H^k)$ in the previous subsection. Moreover, we set $Sp_\psi(f,{\mathcal M},t) $
as $\sum_k(-1)^k Sp(H^k(i_o^*\psi_f\,{\mathcal M}))$.

Here we will construct ${\mathcal M}$ as $j_*{\mathbb V}$, where $j:X\setminus\{o\}\to X$ is the inclusion  and ${\mathbb V}$ is a polarized variation of Hodge structure on $X\setminus \{o\}$ with
 finite cyclic rank one  underlying representation.

\subsection{The construction of ${\mathcal M}$ when the link of $(X,o)$ is a rational homology sphere}\label{ss:MRHS}

\bekezdes Assume that the link $M$ of $(X,o)$ is a rational homology sphere. This means that in any good resolution
$H_1(M, \Z)=L'/L$. Let $\widehat{H}$  be the Pontrjagin dual of $H$, the group of characters
$Hom(H, S^1)$. It is known that there exists   an isomorphism $\theta:H\to \widehat{H}$,
$[l']\mapsto e^{2\pi i(l', \cdot)}$, where $[l']$ denotes the class of $l'\in L'$ in $L'/L$.
In particular, any $l'$ (mod $L$)
gives a character of $H$, hence by the composition
$\pi_1(M)\to H_1(M,\Z)\to S^1$ a finite order character of $\pi_1(M)=\pi_1(X\setminus\{o\})$.
The corresponding flat bundle  can be lifted to the category of polarized variation of Hodge structures on
$X\setminus \{o\}$, this will be denoted by ${\mathbb V}_{[l']}$.

 ${\mathbb V}_{[l']}$ can also be constructed as follows. Let $N$ be the order of
 $[l']$ in $H$. Then the above representation $\rho_{[l']}:\pi_1(X\setminus \{o\})\to H_1(M,\Z)\to \Z_N$
 determines a $\Z_N$--Galois $o$--ramified covering $cov:(X_{[l']},o)\to (X,o)$, where $(X_{[l']},o)$
  is a normal surface singularity, and the restriction  $X_{[l']}\setminus\{o\}\to X\setminus \{o\}$
  is the regular covering associated with $\rho_{[l']}$. This follows from Stein theorem \cite{stein}.
  Then $cov_*\C_{X_{[l']}\setminus \{o\}}=\oplus _{i=1}^N {\mathbb V}^i $, where each ${\mathbb V}^i$  is a variation of Hodge structures  whose underlying monodromy representation is $\rho_{[l']}^i$.
 Our variation we searching for is ${\mathbb V}_{[l']}:={\mathbb V}^1$.

\begin{theorem}\label{th:NS2}\cite{NS} Assume that the link of $(X,o)$ is a $\Q HS^3$. Let $f:(X,o)\to (\C,0)$ be an isolated singularity.  Consider the variation ${\mathbb V}_{[l']}$,  and define the rational numbers  $\beta_v\in[0,1)$  by
$$\rho_{[l']}(-E^*_v)=e^{2\pi i (l', -E^*_v)}=e^{2\pi i \beta_v}\ \ \mbox{for any $v\in \cV$}.$$
For any $a\in\cA$ set $\beta_a=0$.
Define ${\mathcal M}_{[l']}:= j_* {\mathbb V}_{[l']}$, where $j:X\setminus \{o\}\to X$ is the inclusion.  Then
 $$Sp_{[0,1]}(f,{\mathcal M}_{[l']}, t)=TSp_{[0,1]}^\beta.$$
\end{theorem}

Hence, Theorems \ref{th:main_mult} and \ref{th:NS} combined give
\begin{corollary}\label{cor:NS2} Consider the situation of Theorem \ref{th:NS2}.
Let us  fix a good embedded resolution $\phi$ and set $F:={\rm div}_{\phi,E}(f)$.
Furthermore, define $\beta_v=\{-k_v\}$ for $v\in\cV$ and $\beta_a=0$ for $a\in\cA$. Then
    \begin{equation}\label{eq:ext2}
       \sum_{c\in [0,1]}m^*(c)t^c=-\mathfrak{h}^0t^0+ \sum_{c\in [0,1]}m(c)t^c=TSp_{[0,1]}^{\beta}=
       Sp_{[0,1]}(f,{\mathcal  M}_{[-Z_K]}, t).
\end{equation}
\end{corollary}
The covering $cov$ considered above, associated with $-Z_K$, is called the {\it topological canonical
covering} of $(X,o)$ (cf. \cite[6.1.19]{Nkonyv}). (Note that the coverings associated with $[l']$ and
$[-l']$ agree, cf. \cite[Remark 4.1.36]{Nkonyv}.)

\subsection{The construction of ${\mathcal M}$ when  $(X,o)$ is $\Q$--Gorenstein}\label{ss:MRHS2}

 \bekezdes\label{bek:qg} A Weil divisor of $(X,o)$ is called $\Q$--Cartier if its class in ${\rm Cl}(X,o)$ (the local divisor class group of $(X,o)$) has finite order. Its order is called its index.
 If $D$ is a $\Q$--Cartier divisor of index $N$ then it
 determines a $o$--ramified Galois  $\Z_N$--covering $cov:(X_D,o)\to (X,o)$, here $(X_D,o)$ is a normal surface singularity. The covering depends only on the class of $D$ in ${\rm Cl}(X,o)$.  For details see
 \cite[2.49]{km}, \cite{okuma,tomw} or \cite[Proposition 6.1.10]{Nkonyv}.

 Write $\Omega^2_{\tX}=\cO_{\tX}(K_{\tx})$ and set $K_{X}=\phi_*(K_{\tX})$, well--defined in
 ${\rm Cl}(X,o)$. We say that $(X,o)$ is Gorenstein (respectively  $\Q$--Gorenstein)
 if $K_X=0$ (respectively  $K_X$ has finite order in ${\rm Cl}(X,o)$).
 If $(X,o)$ is $\Q$--Gorenstein then the associated cyclic covering is called the `(analytic) canonical covering'.

Assume that the link is a rational homology sphere and  $(X,o)$ is $\Q$--Gorenstein. Then
the coverings $cov:(X_{[l']},o)\to (X,o)$  and  $cov:(X_{K_X},o)\to (X,o)$ agree
\cite[page 179]{Nkonyv}. Thus,
 if we replace the covering $cov:(X_{[l']},o)\to (X,o)$
 from the previous subsection with the covering  $cov:(X_{K_X},o)\to (X,o)$, in a totally similar way we can define
 (via $cov_*\C_{X_{K_X}\setminus \{o\}}$ and  ${\mathbb V}^1$)
  a variation of Hodge structure ${\mathbb V}_{K_X}$ on $X\setminus \{o\}$
  (which agrees with ${\mathbb V}_{[Z_K]}$, and it has the same
  underlying monodromy representation
  and the same $\beta$--coefficients). Hence, if we take ${\mathcal M}_{K_X}:=j_*{\mathbb V}_{K_X}$ then by \cite{NS} and Theorem \ref{th:main_mult}
    \begin{equation}\label{eq:ext3}
       \sum_{c\in [0,1]}m^*(c)t^c=-\mathfrak{h}^0t^0+ \sum_{c\in [0,1]}m(c)t^c=TSp_{[0,1]}^{\beta}=
       Sp_{[0,1]}(f,{\mathcal  M}_{K_X}, t).
\end{equation}

\begin{remark}
Note that the jumping multiplicity/spectral number correspondence
of Corollary  \ref{cor:NS2} is realized for a character (or $\beta$) associated with the cycle $-Z_K$. The general construction of ${\rm TSp}^\beta$, valid for any $\beta$, and the general Theorem \ref{th:NS2}, valid for any character, suggest that there exists a more general correspondence, where one has to generalize the multiplier ideals as well (by replacing the cycle $Z_K$ by a more general cycle). We will return back to such generalizations in a forthcoming manuscript.
\end{remark}

\end{document}